\documentclass[reqno]{amsart} 

\usepackage[bookmarks,colorlinks=true]{hyperref}
\usepackage{amssymb,amsthm}
\usepackage{upgreek}
\usepackage{mathtools}
\usepackage{mathabx}
\usepackage[usenames,pdftex]{color}
\usepackage{thmtools}
\usepackage[capitalise]{cleveref}		
\usepackage{float}
\usepackage{enumitem}

\usepackage{tabularx}
\usepackage{multirow}

\usepackage{color,soul}
\usepackage{tikz}
\usepackage{tikz-cd} 
\usetikzlibrary{shapes}
\usetikzlibrary{shadows}

\usepackage{epigraph}

\makeatletter
\def\thmt@refnamewithcomma #1#2#3,#4,#5\@nil{%
  \@xa\def\csname\thmt@envname #1utorefname\endcsname{#3}%
  \ifcsname #2refname\endcsname
    \csname #2refname\expandafter\endcsname\expandafter{\thmt@envname}{#3}{#4}%
  \fi
}
\makeatother

\definecolor{lightred}{rgb}{1,.60,.60}

\declaretheorem[numberwithin=section]{theorem}
\declaretheorem[sibling=theorem]{proposition}
\declaretheorem[sibling=theorem]{corollary}
\declaretheorem[sibling=theorem]{question}

\declaretheorem[sibling=theorem,style=definition]{definition}

\declaretheorem[sibling=theorem,style=remark]{lemma}

\declaretheorem[sibling=theorem,style=remark,refname={Fact,Facts}]{fact}
\declaretheorem[sibling=theorem,style=remark]{remark}

\declaretheorem[name=Claim,style=remark,numberwithin=theorem,refname={Claim,Claims}]{smallclaim}

\newcommand{\seq}[1]{{\left\langle{#1}\right\rangle}}

\newcommand{\rest}[1]{\! \upharpoonright_{#1}} 
\newcommand{\tth}{{}^{\textup{th}}}
\newcommand{\conc}{\hat{\,\,}}
\newcommand{\andd}{\,\,\,\&\,\,\,}
\newcommand{\Then}{\,\Longrightarrow\,}

\newcommand{\then}{\,\,\rightarrow\,\,}

\DeclareMathOperator{\dom}{dom}

\DeclareMathOperator{\stem}{stem}

\DeclareMathOperator{\Nar}{\texttt{Nar}}

\renewcommand{\setminus}{-}

\newcommand{\w}{\omega}
\newcommand{\s}{\sigma}

\renewcommand{\epsilon}{\varepsilon}

\renewcommand{\le}{\leqslant}
\renewcommand{\ge}{\geqslant}
\renewcommand{\leq}{\leqslant}
\renewcommand{\geq}{\geqslant}
\renewcommand{\preceq}{\preccurlyeq}
\renewcommand{\succeq}{\succcurlyeq}
\renewcommand{\npreceq}{\npreccurlyeq}
\renewcommand{\nsucceq}{\nsucccurlyeq}

\newcommand{\nge}{\ngeqslant}

\newcommand{\Tur}{\textup{\scriptsize T}}
\newcommand{\ML}{\textup{ML}}
\newcommand{\MLR}{\textup{\texttt{MLR}}}
\newcommand{\leb}{\lambda}

\newcommand{\cat}{\widehat{\phantom{\alpha}}}

\newcommand{\wock}{{\w_1^{\textup{ck}}}}

\newcommand{\A}{\mathcal{A}}

\newcommand{\U}{\mathcal{U}}

\newcommand{\T}{\mathcal{T}}

\newcommand{\dbrak}[1]{\ldbrack #1 \rdbrack}

\newcommand{\smallseq}[1]{\langle{#1}\rangle}

\newcommand{\PA}{\textup{PA}{}}

\newcommand{\yes}{\texttt{yes}}
\newcommand{\finh}{\textup{\scriptsize{fin-h}}}

\newcommand{\Access}{S}
\newcommand{\Rat}{\mathbb{Q}}
\newcommand{\given}{\,|\,}

\title{Bad oracles in higher computability and randomness}

\author{Laurent Bienvenu}
\address{LaBRI, 351, cours de la Lib\'eration, 33405 Talence Cedex France}
\email{laurent.bienvenu@computability.fr}
\urladdr{\url{http://www.labri.fr/perso/lbienvenu/}}

\author{Noam Greenberg} 
\address{School of Mathematics and Statistics, Victoria University of Wellington, Wellington, New Zealand}
\email{greenberg@sms.vuw.ac.nz}
\urladdr{\url{http://homepages.ecs.vuw.ac.nz/~greenberg/}}

\author{Benoit Monin}
\address{LACL, Cr\'eteil University, Cr\'eteil, France}
\email{benoit.monin@computability.fr}
\urladdr{\url{https://www.lacl.fr/~benoit.monin}}

\thanks{We would like to thank Joe Miller and Denis Hirschfeldt for useful conversations. 
Bienvenu was supported by the ANR grant RaCAF ANR-15-CE40-0016-01. Greenberg was supported by the Marsden Fund, by a Rutherford Discovery Fellowship from the Royal Society of New Zealand, and by a Turing research fellowship from the Templeton Foundation. Monin was partially supported by the Marsden Fund.}

\begin{document}

\begin{abstract}
Many constructions in computability theory rely on ``time tricks''. In the higher setting, relativising to some oracles shows the necessity of these. We construct an oracle~$A$ and a set~$X$, higher Turing reducible to~$X$, but for which $\Psi(A)\ne X$ for any higher functional~$\Psi$ which is consistent on all oracles. We construct an oracle~$A$ relative to which there is no universal higher $\ML$-test. On the other hand, we show that badness has its limits: there are no higher self-$\PA$ oracles, and for no~$A$ can we construct a higher $A$-c.e.\ set which is also higher $A$-$\ML$-random. We study various classes of bad oracles and differentiate between them using other familiar classes. For example, bad oracles for consistent reductions can be higher ML-random, whereas bad oracles for universal tests cannot. 
\end{abstract}	

\maketitle

\setcounter{tocdepth}{1}

\tableofcontents

\section{Introduction}

\epigraph{``If Croesus goes to war he will destroy a great empire.''}{Herodotus I.92}

Computability theorists are wont to say that ``Everything relativises''. This is not, strictly speaking, true: Shore, for example, showed \cite{Shore:Homogeneity} that upper cones in the Turing degrees are not, in general, isomorphic, or even elementarily equivalent. However all ``natural'' results (compared with techniques of coding models of arithmetic) do remain true if one adds any real parameter, and usually in a uniform way. For example, the construction of a universal Martin-L\"of test relativises to give a uniform oracle test which is universal for every oracle. 

In this paper we show that higher computability does differentiate between classes of oracles, at least when we consider continuous relativisations. Continuous reductions in the higher setting were used by Hjorth and Nies \cite{HjorthNies2007} when they set up the basic framework for higher ML-randomness. Similarly, Chong and Yu \cite{ChongYu} used continuous enumeration operators to define tests for randomness. In \cite{BGMContinuousHigherRandomness}, we argued that a good theory of higher randomness requires using continuous relativisation: from basic results such as van Lambalgen's theorem, to more elaborate, such as the equivalent characterisations of $K$-triviality, continuous relativisations are the ones that give the desired higher analogues. 

However, using continuous relativisations in the higher setting means giving up on time tricks. In classical computability, we sometimes use the fact that the time of computation lies in the same space as the lengths of the sequences we use: $\omega$. We call any use of this equality a \emph{time trick}. Also sometimes, the use of a time trick is just done because it is convenient, but can actually be avoided. An example can be found in the proof that no $\Delta^0_2$ sequence~$Y$ is weakly 2-random. In the usual proof we observe that~$Y$ belongs to the null $\Pi^0_2$ set $\bigcap_{t,n < \omega} \bigcup_{t>s}\{X \mid X \succeq Y_{t} \rest n\}$. This proof clearly uses a time trick, and in this case it is possible to remove it: instead we can consider the set $\A = \bigcap_{n} \bigcup_{s<\omega}\{X \mid X \succeq Y_{s} \rest n\}$. To see that this is null, we note that $\{Y_s\}_{s < \omega} \cup \{Y\}$ is a closed set and as any point in $\A$ is at distance $0$ of this closed set, we then have $\A \subseteq \{Y_s\}_{s < \omega} \cup \{Y\}$ which, as a countable set, has measure $0$.

Sometimes, time tricks cannot be avoided. We show that many straightforward properties of relativisations, which in lower computability rely on time tricks, become false in the higher settings. Like the Pythia misleading the king of Lydia, there are ``bad oracles'' such as:
\begin{itemize}
	\item An oracle which higher computes a set, but cannot compute it using consistent functionals (\cref{thm:bad_oracle_for_consistent_Turing}).
	\item An oracle relative to which there is no universal higher ML-test (\cref{th:no_A_unniversal_mlr_test})
\end{itemize}

We also put limits on how bad oracles can be, for example we show that there are no higher ``self-$\PA$'' oracles (\cref{def:self-PA_oracles}); also, for every oracle~$A$, there is a higher left $A$-c.e.\ sequence which is higher $A$-ML-random (\cref{thm:higher_left_c.e._random}). Both directions involve novel techniques. To show that oracles cannot be very bad, we sometimes need non-uniform arguments which do exhibit the distinct behaviour of different classes of oracles. In the other direction, the construction of a variety of bad oracles uses a new technique, using ``treesh-bones'', which we introduce in this paper. Finally we compare different classes of bad oracles and their properties with respect to classes such as higher Turing complete oracles, or higher random oracles. 

\smallskip

We remark that J.~Miller and M.~Soskova (unpublished) have recently studied bad oracles in a different setting, namely the enumeration degrees, in which for example there are self-PA degrees. It seems however that ``the symptoms are similar but the disease is different.'' In particular in their setting it is not the failure of time tricks which is the origin of badness, but rather the lack of a canonical form (enumeration) of the oracle. 

In a planned sequel to this paper, with J.~Miller, we show that like in the enumeration degree setting, some bad oracles can differentiate between different forms of higher ML-randomness.

\section{Background, notation, and treesh-bones}
\label{sec:background}

We assume that the reader is familiar with the basics of higher computability. A standard references is~\cite{Sacks1990}.

\subsection{Higher effective topology and continuous computation}

We recall the notions of higher effective topology which were introduced in \cite{BGMContinuousHigherRandomness}. Recall that the inspiration for the study of higher computability is the analogy between the classes $\Sigma^0_1$ and $\Pi^1_1$, most lucidly exhibited by the fact that $\Pi^1_1$ subsets of~$\w$ are preciesly the subsets of~$\w$ which are $\Sigma_1$-definable in the structure $(L_{\wock};\in)$; in the language of higher computability, they are $\wock$-c.e.\ sets. Informally, they are enumerated by effective processes which run in $\wock$-many steps. 

This analogy extends to defining sets of reals, inspired by the following equivalence for open sets $\+W\subseteq 2^\w$:
\begin{enumerate}
	\item $\+W$ is $\Pi^1_1$;
	\item $\+W = [W]^\prec$ for a $\Pi^1_1$ set $W\subseteq 2^{<\w}$. 
\end{enumerate}
Here $[W]^\prec = \left\{ X\in 2^\w \,:\, (\exists \s \prec X)\,\,\s\in W  \right\}$ is the set of reals \emph{generated} or \emph{described} by the set of strings~$W$. Continuing the terminology used in \cite{BGMContinuousHigherRandomness}, the $\Pi^1_1$ open sets are the \emph{higher effectively open} sets or \emph{higher c.e.\ open} sets. 

\medskip
   
When we come to discuss relativisations, there are two options for defining, for an oracle $A\in 2^\w$, the notions of higher c.e.\ sets and higher c.e.\ open sets relative to $A$:
\begin{itemize}
	\item sets which are $\Pi^1_1(A)$; 
	\item $A$-sections of $\Pi^1_1$ open sets. 
\end{itemize}
These notions are very different. One of the main points of \cite{BGMContinuousHigherRandomness} is that the latter notion is the one which is useful for us when studying higher randomness and genericity. Here we see that it also allows us to distinguish between different kinds of oracles. For a set $W\subseteq 2^{<\w}\times \w$ and $A\in 2^\w$ let 
\[
	W^A = \left\{ n \,:\,  (\exists \tau\prec A)\,(\tau,n)\in W \right\}.
\]

\begin{definition} \label{def:relative_higher_ce}
	Let $A\in 2^\w$. 
	\begin{enumerate}[label=(\alph*)]
		\item A subset of $\w$ is \emph{higher $A$-c.e.}\ if it is of the form $W^A$ for some  $\Pi^1_1$ set $W\subseteq 2^{<\w}\times \w$
		\item A subset of~$2^\w$ is \emph{higher $A$-c.e.\ open} if it is of the form $[V]^\prec$ for some higher $A$-c.e.\ set $V\subseteq 2^{<\w}$. 
	\end{enumerate}
\end{definition}
Observe that the higher $A$-c.e.\ sets are the $A$-sections of higher c.e.\ open subsets of $2^\w\times \w$, and the higher $A$-c.e.\ open sets are the $A$-sections of higher c.e.\ open subsets of $2^\w\times 2^\w$. A $\Pi^1_1$ set $W\subseteq 2^{<\w}\times \w$ is called a \emph{higher enumeration operator} or a \emph{higher oracle c.e.\ set}. Similarly, If $\+W\subseteq 2^\w\times 2^\w$ is higher effectively open then we also think of it as a \emph{higher oracle c.e.\ open set}; we sometimes use the same terminology to describe a $\Pi^1_1$ set $W\subseteq 2^{<\w}\times 2^{<\w}$ generating~$\+W$. For such a set we write $\+W^A$ for the $A$-section of~$\+W$. 

\smallskip

Using open sets, we can extend higher continuous relativisation up the arithmetic hierarchy. In this paper though we will restrict ourselves to relativising higher effectively open sets and their complements, the higher effectively closed (or co-c.e.) sets. 


\medskip

Having tackled open sets, we can now deal with partial continuous maps. A code for a partial continuous function $F\colon 2^\w\to 2^\w$ is a set $\Phi\subseteq 2^{<\w}\times 2^{<\w}$ such that for $X\in \dom F$, 
\[
	F(X) = \bigcup \left\{ \s \,:\,  (\exists \tau\prec X) \,\,(\tau,\s)\in \Phi \right\}. 
\]
The set~$\Phi$ is often called a \emph{functional}, and we usually abuse notation by using~$\Phi$ to denote both the functional and the induced partial continuous function~$F$. When $\Phi$ is~$\Pi^1_1$, this gives us a higher reduction procedure, so we make the following definition:

\begin{definition} \label{def:higher_relative_computability}
	Let $X,Y\in 2^\w$. We say that $Y$ is \emph{higher $X$-computable}, and write $Y\le_{\wock\Tur} X$, if $Y = \Phi(X)$ for some $\Pi^1_1$ functional~$\Phi$. 
\end{definition}
Note that $Y\le_{\wock\Tur} X$ if and only if the set of final initial segments of~$Y$ is higher $X$-c.e. With this definition, familiar facts about relative computability hold in the higher setting, for example: $Y\le_{\wock\Tur} X$ if and only if both~$Y$ and its complement are higher $X$-c.e. When we think of them as codes for higher partial continuous functions, we refer to $\Pi^1_1$ sets $\Phi\subseteq 2^{<\w}\times 2^{<\w}$ as \emph{higher Turing functionals}.

\subsection{Consistency and uniformity} 
\label{sub:consistency_and_uniformity}

Let $\Phi$ be a higher Turing functional, and let $X\in 2^\w$. There are two possible reasons that $\Phi(X)$ may fail to be an element of $2^\w$:
\begin{enumerate}[label=(\roman*)]
	\item Partiality: $\Phi(X)$ is finite, i.e.\ an element of $2^{<\w}$; or
	\item Inconsistency: $\Phi(X)$ is not a function, which happens if there are $\tau,\tau'$, finite initial segments of~$X$, and incompatible $\s,\s'$ such that $(\tau,\s), (\tau',\s')\in \Phi$. 
\end{enumerate}
It is of course possible that $\Phi$ is consistent and total on some oracles but not on others. \Cref{def:higher_relative_computability} allows all possible functionals. We could define a stronger reducibility, for which we require the functional~$\Phi$ to be consistent. The distinction between the two does not occur when considering c.e.\ functionals: using a time trick, one can easily uniformly transform any c.e.\ functional~$\Phi$ into another c.e.\ functional~$\Psi$ which is consistent everywhere and such that whenever $\Phi(X) = Y$, we also have $\Psi(X) = Y$. This construction uses a time trick, and fails when considering $\Pi^1_1$ functionals. In \cite{BGMContinuousHigherRandomness} we argued that requiring consistent functionals was in most cases too restrictive. We will show (in \cref{section:separation_fin_h_higher_turing}) that there is no way around that: there exists some oracle~$X$ which higher Turing computes some~$Y$, but which does not compute it via any functional which is consistent everywhere.

\subsection{Failure of time tricks: an easy example}

We shall now see a first example where the use of a time trick cannot be removed. Any open set with a $\Sigma^0_1$ description also has a $\Delta^0_1$ description. Indeed, for any $\Sigma^0_1$ set of strings $W$, we can define a $\Delta^0_1$ set $V$ generating~$\+W$ by enumerating~$\s$ in~$V$ at stage~$|\s|$ if some prefix of~$\s$ is enumerated in~$W$ at stage~$\s$.

The proof of the previous paragraph clearly uses a time trick. We shall now see that there are some open sets with a $\Pi^1_1$ description that do not have a~$\Delta^1_1$ description. We start by proving the following:

\begin{proposition} \label{th:prefix_free_failure}
There is a~$\Pi^1_1$ open set which is not generated by any~$\Pi^1_1$ prefix-free set of strings. 
\end{proposition}

\begin{proof}
Let $\{W_e\}_{e \in \omega}$ be a list of all $\Pi^1_1$ set of strings. Let $\{\s_e\}_{e \in \omega}$ be an infinite sequence of pairwise incomparable strings (for example let $\s_e = 0^e1$). Working in $L_\wock$, we enumerate a~$\Pi^1_1$ set of strings~$V$ such that for all~$e$, if $W_e$ is prefix-free, then $[V]^\prec \neq [W_e]^\prec$.

For each~$e$ we let~$A_e$ be a set of strings extending~$\s_e$ which is dense along $\s_e \cat 0^\infty$ but which contains no prefix of $\s_e \cat 0^\infty$, for example $A_e = \left\{ \s_e\cat 0^n1 \,:\, n<\w  \right\}$. At stage~$0$ of our construction we start with $V_0 = \bigcup_e A_e$.  Then for any stage~$s$, and substage~$e$, we check if both $[A_e]^\prec \subseteq [W_{e, s}]^\prec$ and $\s_e \cat 0^\infty \notin [W_{e, s}]^\prec$. If so, then we enumerate~$\s_e$ into~$V$ at stage~$s$.

We now claim that if $W_e$ is prefix-free, then $[V]^\prec \neq [W_e]^\prec$. If $[V]^\prec =  [W_e]^\prec$, in particular we have $[A_e]^\prec \subseteq  [W_e]^\prec$. If so, then by compactness, for each string $\tau$ in $A_e$, there are only finitely many strings in $W_e$ whose union of corresponding cylinders covers $[\tau]$. Also by admissibility of~$\wock$ (equivalently, the $\Sigma^1_1$-bounding principle), as~$A_e$ is computable, there is a stage $s < \wock$ at which we already have $[A_e]^\prec \subseteq  [W_{e, s}]^\prec$; let~$s$ be the least such stage. By construction, if $\s_e \cat 0^\infty \in [W_{e, s}]^\prec$ then $\s_e \cat 0^\infty \notin [V]$, in which case the point $\s_e\cat 0^\infty$ shows that $[V]^\prec \neq [W_e]^\prec$.

On the other hand, if at stage $s$ we have $\s_e \cat 0^\infty \notin [W_{e, s}]^\prec$, then $\s_e$ is enumerated in $V$ at stage $s$. If $\s_e \cat 0^\infty \notin [W_{e}]^\prec$ then again we get $[V]^\prec \neq [W_e]^\prec$. Otherwise, a prefix $\tau$ of $\s_e \cat 0^\infty$ will be enumerated into $W_{e}$ after stage $s$. But then, as already at stage~$s$ we have that $[W_{e}]^\prec$ covers $[A_e]^\prec$ without containing $\s_e \cat 0^\infty$, and as~$A_e$ is dense along $\s_e \cat 0^\infty$, there is necessarily an extension of~$\tau$ which is already in~$W_e$ at stage $s$. Therefore~$W_e$ is not prefix-free.
\end{proof}

\begin{corollary}
There is a $\Pi^1_1$-open set which is not generated by any $\Delta^1_1$ set of strings.
\end{corollary}

\begin{proof}
For any $\Delta^1_1$ set of strings $W$ there is a prefix-free $\Delta^1_1$ set of strings~$V$ with $[W]^\prec = [V]^\prec$, namely the set of minimal strings in~$W$. 
\end{proof}

\subsection{Treesh-bones: motivation}

To motivate the construction framework that we will soon describe, we consider, informally, how we would go about showing that there are $Y\le_{\wock\Tur} X$ such that for no everywhere consistent functional do we have $\Phi(X)=Y$. The idea is similar to the proof of \cref{th:prefix_free_failure}. Let~$\Phi$ be a consistent functional that we want to defeat. Consider the standard ``fish-bone'' with spine $0^\infty$ and ribs $0^n1$ for $n<\w$. We design our functional~$\Psi$. To start, we let~$\Psi$ map the string $1$ to $1$, and declare that currently, $X$ extends $1$. Now we wait for the opponent, playing~$\Phi$, to respond. There are three possibilities:
\begin{enumerate}
	\item The opponent never responds: there is no~$\tau$, compatible with $1$, which is mapped by~$\Phi$ to $1$ (or some extension of $1$). In this case we declare that $X$ and~$Y$ both extend~$1$, and win against~$\Phi$. We can proceed to defeat other consistent functionals. 
	\item The opponent responds by mapping the empty string to $1$: in this case we win by declaring that both~$X$ and~$Y$ extend~$0$. 
	\item The opponent maps some extension of $1$ to $1$. In this case we declare that for now, $X$ extends the string~$01$, and map~$01$ to~$1$; we repeat. 
\end{enumerate}
If the third outcome keeps occurring, then we successivley map each $0^n1$ to $1$, and the opponent maps $0^n1$ (or some extension) to 1. The opponent's response is a~$\Sigma_1$ event; admissibility of~$\wock$ implies that some stage $s<\wock$ bounds the time of all these responses. When we get to that stage -- when we see that for all~$n$, the opponent mapped some extension of $0^n1$ to $1$ -- we can declare that $X = 0^\infty$ and that~$Y$ extends 0. We get to map the empty string to 0, because~$\Psi$ is not required to be conistent, only consistent on~$X$. The opponent has no such recourse; he cannot map any initial segment $0^n$ to 0, as each has an extension which is mapped to 1. Thus we defeat~$\Phi$. 

The problem, of course, is that by that stage we have commited to $X = 0^\infty$ and have no room to move to defeat the next functional. The idea of the treesh-bone is to make sure that we can carry out this construction, but when the interesting outcome occurs, we are left not with just a spine but with a full binary tree on which we can play the same game with the next functional.

\subsection{The perfect treesh-bone}
\label{subsec:treesh-bone}

We now present a general framework, that will be used for several of this paper's constructions.

For a tree $T \subseteq 2^{<\omega}$, recall that $\s \in T$ is a branching node of $T$ if $\s \cat 0$ and $\s \cat 1$ are both in $T$, and recall that the stem of $T$, denoted by $\stem(T)$, is the shortest branching node of $T$. Given~$T$, we want to obtain both a prefect subtree, which we call $\Nar(T)$, together with countably many nodes $\{\s_i\}_{i \in \omega}$ of~$T$ which do not belong to $\Nar(T)$, but which are dense along any path of $\Nar(T)$. We now formally describe how we achieve this.

Let $T \subseteq 2^{<\omega}$ be perfect. Let $\psi_T\colon 2^{<\omega}\to T$ be the order-preserving map whose range is the collection of branching nodes of~$T$. That is, $\psi_T(\epsilon) = \stem(T)$ (the shortest branching node of~$T$) and for all $\s\in 2^{<\omega}$ and $i \in \{0,1\}$, $\psi_T(\s \cat i)$ is the next branching node in $T$ above $\psi_T(\s) \cat i$. 

Let $\s_0(T), \s_1(T), \dots$ be an enumeration of all strings of the form $\psi_T(\s \cat 1)$ for strings $\s$ of \emph{odd} length such that for all even $k<|\s|$, $\s(k)=0$. That is, the nodes $\s_k(T)$ are the minimal nodes on~$T$ of the form $\psi_T(\s\cat 1)$ for~$\s$ of odd length. We let $\Nar(T)$ (the narrow subtree of~$T$) be the result of removing the strings~$\s_k(T)$ from~$T$ (and keeping a perfect subtree with no dead ends). Also for any~$k$ we let $T\ldbrack k \rdbrack$ denote $T\rest {\s_k(T)}$, that is, the collection of strings of~$T$ comparable with $\s_k(T)$ (also known as the full subtree of~$T$ issuing from $\s_k(T)$). \Cref{fig:treesh-bone} illustrate these definitions. 
\begin{figure}[h]
\centering

\begin{tikzpicture}[xscale=0.11, yscale=.07]
\draw[thick,color=red] (50.00,0.0) -- (25.00,13.00);
\draw[thick,color=red] (50.00,0.0) -- (75.00,13.00);
\draw[thick,color=red] (25.00,13.0) -- (25.00,26.00);
\draw[thick,color=blue,line width=.5mm] (25.00,13.0) -- (37.50,26.00);
\draw[thick,color=red] (75.00,13.0) -- (75.00,26.00);
\draw[thick,color=blue,line width=.5mm] (75.00,13.0) -- (87.50,26.00);
\draw[thick,color=red] (25.00,26.0) -- (12.50,39.00);
\draw[thick,color=red] (25.00,26.0) -- (37.50,39.00);
\draw[thick,color=red] (75.00,26.0) -- (62.50,39.00);
\draw[thick,color=red] (75.00,26.0) -- (87.50,39.00);
\draw[thick,color=red] (12.50,39.0) -- (12.50,52.00);
\draw[thick,color=blue,line width=.5mm] (12.50,39.0) -- (18.75,52.00);
\draw[thick,color=red] (37.50,39.0) -- (37.50,52.00);
\draw[thick,color=blue,line width=.5mm] (37.50,39.0) -- (43.75,52.00);
\draw[thick,color=red] (62.50,39.0) -- (62.50,52.00);
\draw[thick,color=blue,line width=.5mm] (62.50,39.0) -- (68.75,52.00);
\draw[thick,color=red] (87.50,39.0) -- (87.50,52.00);
\draw[thick,color=blue,line width=.5mm] (87.50,39.0) -- (93.75,52.00);
\draw[thick,color=red] (12.50,52.0) -- (6.25,65.00);
\draw[thick,color=red] (12.50,52.0) -- (18.75,65.00);
\draw[thick,color=red] (37.50,52.0) -- (31.25,65.00);
\draw[thick,color=red] (37.50,52.0) -- (43.75,65.00);
\draw[thick,color=red] (62.50,52.0) -- (56.25,65.00);
\draw[thick,color=red] (62.50,52.0) -- (68.75,65.00);
\draw[thick,color=red] (87.50,52.0) -- (81.25,65.00);
\draw[thick,color=red] (87.50,52.0) -- (93.75,65.00);
\draw[thick,color=red] (6.25,65.0) -- (6.25,78.00);
\draw[thick,color=blue,line width=.5mm] (6.25,65.0) -- (9.38,78.00);
\draw[thick,color=red] (18.75,65.0) -- (18.75,78.00);
\draw[thick,color=blue,line width=.5mm] (18.75,65.0) -- (21.88,78.00);
\draw[thick,color=red] (31.25,65.0) -- (31.25,78.00);
\draw[thick,color=blue,line width=.5mm] (31.25,65.0) -- (34.38,78.00);
\draw[thick,color=red] (43.75,65.0) -- (43.75,78.00);
\draw[thick,color=blue,line width=.5mm] (43.75,65.0) -- (46.88,78.00);
\draw[thick,color=red] (56.25,65.0) -- (56.25,78.00);
\draw[thick,color=blue,line width=.5mm] (56.25,65.0) -- (59.38,78.00);
\draw[thick,color=red] (68.75,65.0) -- (68.75,78.00);
\draw[thick,color=blue,line width=.5mm] (68.75,65.0) -- (71.88,78.00);
\draw[thick,color=red] (81.25,65.0) -- (81.25,78.00);
\draw[thick,color=blue,line width=.5mm] (81.25,65.0) -- (84.38,78.00);
\draw[thick,color=red] (93.75,65.0) -- (93.75,78.00);
\draw[thick,color=blue,line width=.5mm] (93.75,65.0) -- (96.88,78.00);
\draw[thick,color=red] (6.25,78.0) -- (3.13,91.00);
\draw[thick,color=red] (6.25,78.0) -- (9.38,91.00);
\draw[thick,color=red] (18.75,78.0) -- (15.63,91.00);
\draw[thick,color=red] (18.75,78.0) -- (21.88,91.00);
\draw[thick,color=red] (31.25,78.0) -- (28.13,91.00);
\draw[thick,color=red] (31.25,78.0) -- (34.38,91.00);
\draw[thick,color=red] (43.75,78.0) -- (40.63,91.00);
\draw[thick,color=red] (43.75,78.0) -- (46.88,91.00);
\draw[thick,color=red] (56.25,78.0) -- (53.13,91.00);
\draw[thick,color=red] (56.25,78.0) -- (59.38,91.00);
\draw[thick,color=red] (68.75,78.0) -- (65.63,91.00);
\draw[thick,color=red] (68.75,78.0) -- (71.88,91.00);
\draw[thick,color=red] (81.25,78.0) -- (78.13,91.00);
\draw[thick,color=red] (81.25,78.0) -- (84.38,91.00);
\draw[thick,color=red] (93.75,78.0) -- (90.63,91.00);
\draw[thick,color=red] (93.75,78.0) -- (96.88,91.00);
\draw[thick,color=red] (3.13,91.0) -- (3.13,104.00);
\draw[thick,color=blue,line width=.5mm] (3.13,91.0) -- (4.69,104.00);
\draw[thick,color=red] (9.38,91.0) -- (9.38,104.00);
\draw[thick,color=blue,line width=.5mm] (9.38,91.0) -- (10.94,104.00);
\draw[thick,color=red] (15.63,91.0) -- (15.63,104.00);
\draw[thick,color=blue,line width=.5mm] (15.63,91.0) -- (17.19,104.00);
\draw[thick,color=red] (21.88,91.0) -- (21.88,104.00);
\draw[thick,color=blue,line width=.5mm] (21.88,91.0) -- (23.44,104.00);
\draw[thick,color=red] (28.13,91.0) -- (28.13,104.00);
\draw[thick,color=blue,line width=.5mm] (28.13,91.0) -- (29.69,104.00);
\draw[thick,color=red] (34.38,91.0) -- (34.38,104.00);
\draw[thick,color=blue,line width=.5mm] (34.38,91.0) -- (35.94,104.00);
\draw[thick,color=red] (40.63,91.0) -- (40.63,104.00);
\draw[thick,color=blue,line width=.5mm] (40.63,91.0) -- (42.19,104.00);
\draw[thick,color=red] (46.88,91.0) -- (46.88,104.00);
\draw[thick,color=blue,line width=.5mm] (46.88,91.0) -- (48.44,104.00);
\draw[thick,color=red] (53.13,91.0) -- (53.13,104.00);
\draw[thick,color=blue,line width=.5mm] (53.13,91.0) -- (54.69,104.00);
\draw[thick,color=red] (59.38,91.0) -- (59.38,104.00);
\draw[thick,color=blue,line width=.5mm] (59.38,91.0) -- (60.94,104.00);
\draw[thick,color=red] (65.63,91.0) -- (65.63,104.00);
\draw[thick,color=blue,line width=.5mm] (65.63,91.0) -- (67.19,104.00);
\draw[thick,color=red] (71.88,91.0) -- (71.88,104.00);
\draw[thick,color=blue,line width=.5mm] (71.88,91.0) -- (73.44,104.00);
\draw[thick,color=red] (78.13,91.0) -- (78.13,104.00);
\draw[thick,color=blue,line width=.5mm] (78.13,91.0) -- (79.69,104.00);
\draw[thick,color=red] (84.38,91.0) -- (84.38,104.00);
\draw[thick,color=blue,line width=.5mm] (84.38,91.0) -- (85.94,104.00);
\draw[thick,color=red] (90.63,91.0) -- (90.63,104.00);
\draw[thick,color=blue,line width=.5mm] (90.63,91.0) -- (92.19,104.00);
\draw[thick,color=red] (96.88,91.0) -- (96.88,104.00);
\draw[thick,color=blue,line width=.5mm] (96.88,91.0) -- (98.44,104.00);
\end{tikzpicture}

\caption{The treeshbone. The \textcolor{blue}{blue nodes} correspond to nodes $\s_k(T)$. The \textcolor{red}{red subtree} corresponds to $\Nar(T)$.}
\label{fig:treesh-bone}
\end{figure}
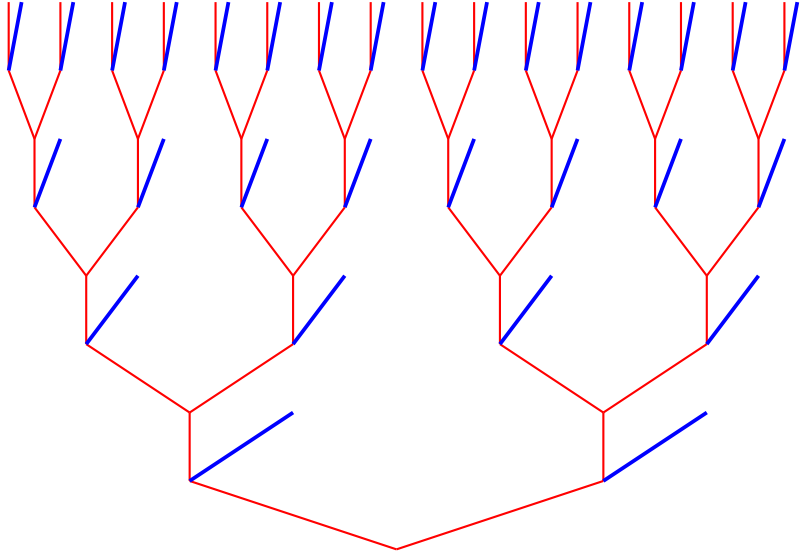

Let us make a few remarks which will be widely used without explicit mention in what follows.

\begin{enumerate}[label=(\alph*)]
\item For any tree $T$ and any $k$ we have $\s_k(T) = \stem(T \ldbrack k \rdbrack)$;
\item For any tree $T$ and any $k$ we have $\stem(T) \prec \stem(T \ldbrack k \rdbrack)$;
\item For any tree $T$ we have $\stem(T) = \stem(\Nar(T))$.
\end{enumerate}

\subsubsection{The tree of trees} \label{subsec:tree_of_trees}

For constructions in this paper, we will work within what can be considered a ``tree of trees''. We define a subset $\T$ of the set of computable trees $T\subseteq 2^{<\w}$. This is the set generated by starting with the full tree~$2^{<\w}$ and closing under taking narrow subtrees and the full subtrees $T\dbrak{k}$ for $k<\w$. For two trees $T_1,T_2$ in~$\T$, we say that $T_2$ extends $T_1$ (and write $T_1 \preceq T_2$) if $T_2 \subseteq T_1$. In addition if $T_1 \neq T_2$ we write $T_1 \prec T_2$. We illustrate the tree of trees by the following picture:

\begin{figure}[H]

\centering

\begin{tikzpicture}[xscale=0.11, yscale=.07]
\node (T) at (50,0) {$T$};
\node (T1) at (25,20) {$\Nar(T)$};
\node (T2) at (75,20) {$T \ldbrack i \rdbrack...$};

\node (T11) at (12.5,40) {$\Nar()$};
\node (T12) at (37.5,40) {$\ldbrack i \rdbrack...$};
\node (T21) at (62.5,40) {$\Nar()$};
\node (T22) at (87.5,40) {$\ldbrack i \rdbrack...$};

\node (T111) at (6.25,60) {$\Nar()$};
\node (T112) at (18.75,60) {$\ldbrack i \rdbrack...$};
\node (T121) at (31.25,60) {$\Nar()$};
\node (T122) at (43.75,60) {$\ldbrack i \rdbrack...$};
\node (T211) at (56.25,60) {$\Nar()$};
\node (T212) at (68.75,60) {$\ldbrack i \rdbrack...$};
\node (T221) at (81.25,60) {$\Nar()$};
\node (T222) at (93.75,60) {$\ldbrack i \rdbrack...$};

\draw[thick] (T) -- (T1);
\draw[thick] (T) -- (72,16);
\draw[thick] (T) -- (74,16);
\draw[thick] (T) -- (76,16);


\draw[thick] (T1) -- (T11);
\draw[thick] (T1) -- (35,36);
\draw[thick] (T1) -- (37,36);
\draw[thick] (T1) -- (39,36);

\draw[thick] (T2) -- (T21);
\draw[thick] (T2) -- (84,36);
\draw[thick] (T2) -- (86,36);
\draw[thick] (T2) -- (88,36);


\draw[thick] (T11) -- (T111);
\draw[thick] (T11) -- (16.75,56);
\draw[thick] (T11) -- (18.75,56);
\draw[thick] (T11) -- (20.75,56);

\draw[thick] (T12) -- (T121);
\draw[thick] (T12) -- (41.75,56);
\draw[thick] (T12) -- (43.75,56);
\draw[thick] (T12) -- (45.75,56);

\draw[thick] (T21) -- (T211);
\draw[thick] (T21) -- (66.75,56);
\draw[thick] (T21) -- (68.75,56);
\draw[thick] (T21) -- (70.75,56);

\draw[thick] (T22) -- (T221);
\draw[thick] (T22) -- (91.75,56);
\draw[thick] (T22) -- (93.75,56);
\draw[thick] (T22) -- (95.75,56);
\end{tikzpicture}

\caption{The tree of trees.}
\end{figure}
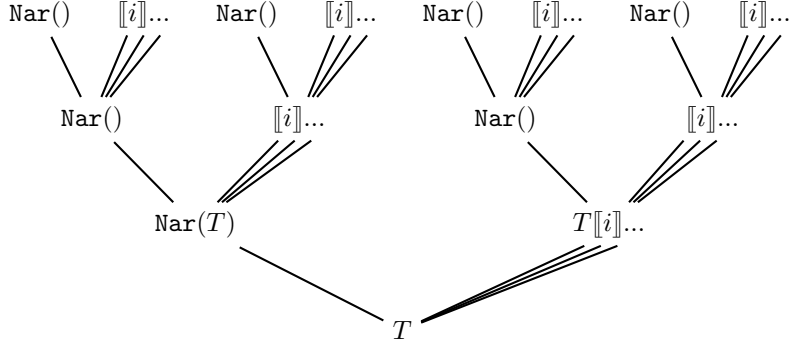

Note that $T_2$ extends~$T_1$ exactly when~$T_2$ is obtained from~$T_1$ by applying finitely many operations of taking the narrow subtree or a full subtree $T\dbrak{k}$. If~$T_1$ and~$T_2$ are incomparable in~$\+T$, then the closed sets $[T_1]$ and~$[T_2]$ are disjoint. Indeed, for any~$T\in \+T$, the collection of $[S]$ for the immediate successors~$S$ of~$T$ on~$\+T$ forms a partition of~$[T]$. 

A \emph{path} in the tree of trees is a sequence $\seq{T_n}$ such that $T_0  = 2^{<\w}$ and for all~$n$, $T_{n+1}$ is an immediate successor of~$T_n$ on~$\+T$. If from some point onwards, the path only takes narrow subtrees -- for some~$n$, for all $m>n$, $T_{m+1} = \Nar(T_m)$ -- then $\bigcap_n T_n$ is a hyperarithmetic perfect set. Otherwise, $\bigcap_n [T_n]$ is a singleton $\{A\}$ and $A = \bigcup_n \stem(T_n)$. In this case we call the path $\seq{T_n}$ \emph{shrinking}.

\subsection{Approximating a sequence of trees} 

A bad oracle we construct will be in the intersection $\bigcap [T_n]$ determined by a path $\seq{T_n}$ in~$\+T$. 
The path $\seq{T_n}$ will usually be $O$-computable, where by~$O$ we denote any complete $\Pi^1_1$ set of natural numbers (usually referred to as \emph{Kleene's $O$}). To show that~$A$ computes some objects, we will need to build higher functionals, which means that we will need to perform some $\wock$-computable construction, rather than directly appeal to the oracle~$O$. Thus in our construction we will approximate the sequence $\seq{T_n}$. 

Here we rely on the following characterisation of $O$-computable sets --  the higher analogue of Shoenfield's limit lemma, which was investigated in \cite{BGMContinuousHigherRandomness}. The following are equivalent for $X\in 2^\w$:
\begin{enumerate}
	\item $X\le_\Tur O$;
	\item $X\le_{\wock \Tur} O$; 
	\item $X$ has a $\wock$-computable approximation: a $\wock$-computable sequence $\seq{X_s}_{s<\wock}$ which converges to~$X$ in the sense that for all $k<\w$, $\left\{ s<\wock \,:\, X_s(k)\ne X(k)  \right\}$ is bounded below~$\wock$. 
\end{enumerate}
We call such reals~$X$ \emph{higher $\Delta^0_2$}. 

\smallskip

Thus, during our construction we will, at every stage~$s<\wock$, build a sequence $\seq{T_{n,s}}_{n<\w}$ -- a path through the tree of trees -- and we will ensure that for each~$n$, the sequence~$\seq{T_{n,s}}_{s<\wock}$ will eventually stabilise. 

\smallskip

The question, though, is \emph{why} the sequence of trees will indeed stabilise to a desired sequence $\seq{T_n}$. Here the dynamics of the construction matter, and we will use the admissibility of~$\wock$. A typical argument will go as follows. Suppose that for all $s\ge s_0$ we have $T_{n,s} = T_n$. Then we ensure that for $t>s\ge s_0$,
\begin{itemize}
	\item if $T_{n+1,t} = T_n\dbrak{k}$ for some~$k<\w$, then $T_{n+1,s} = T_n\dbrak{m}$ for some $m\le k$; and 
	\item if $T_{n+1,s} = \Nar(T_n)$ then $T_{n+1,t} = \Nar(T_n)$ as well. 
\end{itemize}
So once we choose $\Nar(T_n)$, we are guaranteed stabilisation. Admissibility ensures that if for unboundedly many~$k$ we at some point choose $T_{n+1,s} = T_{n}\dbrak{k}$, then there will be some $t<\wock$ by which we will have already seen the unboundedly many~$k$ chosen; we are then forced to choose $T_{n+1,t} = \Nar(T_n)$.

\subsection{Left-c.e.\ strategies} 
\label{subsec:formalising_the_argument}

We describe a way to formalise and generalise this argument. This generalisation will have two parts:
\begin{itemize}
	\item Defining a mapping from \emph{strategies} to trees; and
	\item Approximating a \emph{higher left-c.e.} path of strategies.   
\end{itemize}

\subsubsection{Mapping strategies to trees} 

In most of our constructions we will fix some $\alpha_*<\wock$, usually $\alpha_*=\w+1$ and sometimes $\alpha_* = \w+2$. A \emph{strategy} will be an element of $(\alpha_*)^{<\w}$. We will then assign strategies to trees. For example, in the simpler constructions, we will have $\alpha_* = \w+1$; the outcome~$k$ will corrspond to the operation $T\mapsto T\dbrak{k}$, while the outcome~$\w$ will correspond to the operation $T\mapsto \Nar(T)$. We will then define~$T_\alpha\in \+T$ for a strategy~$T$ by induction on $|\alpha|$: $T_{\epsilon} = 2^{<\w}$, $T_{\alpha\cat k} = T_\alpha\dbrak{k}$, and $T_{\alpha\cat \w} = \Nar(T_\alpha)$. 

Then, at every stage~$s$ of the construction, we will define a path $\xi_s\in (\alpha_*)^\w$ of strategies; $\{\xi_s\rest{n}\,:\, n<\w\}$ will be the collection of strategies accessible at stage~$s$. Then of course we define $T_{n,s} = T_{\xi_s\rest{n}}$.

\subsubsection{The space $(\wock)^\w$} 

The oracles we deal with are elements of Cantor space $2^\w$. In standard computability though we often extend our notions of computability to elements of Baire space $\w^\w$. In the same way that we handle functionals which define maps from Cantor space to itself, we can have c.e.\ functionals which contain pairs $(\s,\tau)$ of finite sequences of natural numbers. These functionals define partial maps from Baire space to itself. 

In higher computability, exactly the same idea can be extended to the space $(\wock)^\w$. In our definitions of higher Turing functionals, for the sake of continuity, it is important that the maps are determined by \emph{finite} initial segments, and so that the objects on which the continuous maps are defined must have length~$\w$. However it is not that important that the \emph{entries} of these objects are binary digits, or natural numbers. We can allow functionals which are $\wock$-c.e.\ subsets of $(\wock)^{<\w}\times (\wock)^{<\w}$, and they define partial maps from $(\wock)^\w$ to itself. Thus, we extend the relation $\le_{\wock\Tur}$ to $\w$-sequences of computable ordinals. This was done in \cite{BGMContinuousHigherRandomness} when analysing the notion $\MLR[O]$, the partial relativisation of ML-randomness to Kleene's~$O$.

\subsubsection{Higher left-c.e.\ sequences} 

The higher limit lemma extends to show that for $X\in (\wock)^\w$, $X\le_{\wock\Tur} O$ if and only if there is a $\wock$-computable sequence $\seq{X_s}_{s<\wock}$ with each $X_s\in (\wock)^\w$ and $\seq{X_s}$ converging to~$X$ in that for all~$n$, $\left\{ s<\wock \,:\,  X_s(n)\ne X(n) \right\}$ is bounded below~$\wock$; so we call such sequences higher~$\Delta^0_2$ as well.  

In \cite{BGMContinuousHigherRandomness}, several subclasses of the higher~$\Delta^0_2$ reals have been analysed. Each subclass is determined by putting restrictions on the type of $\wock$-computable approximations for elements in the subclass. For example, we can require that on each input there are only finitely many mind-changes. A prominent subclass is that of the \emph{higher left-c.e.} sequences. Just like Cantor space and Baire space, the natural ordering on~$\wock$ induces the lexicographic ordering on $(\wock)^\w$. The following are equivalent for $X\in (\wock)^\w$: 
\begin{enumerate}
	\item $\left\{ \s\in (\wock)^{<\w} \,:\,  \s<X \right\}$ is $\wock$-c.e.; 
	\item $X$ has a $\wock$-computable approximation $\seq{X_s}$ with $X_s\le X_t$ when $s<t$. 
\end{enumerate}
Such sequences are called higher left-c.e. The implication (1)$\to$(2) uses admissibility, namely, for all~$n$, all $\s<X$ of length~$n$ are enumerated by some stage. Generally, if we are given a lexicographically increasing sequence $\seq{X_s}$ there is no guarantee that it converges to some sequence; but if the~$X_s$ are uniformly bounded then we do get convergence. Namely:
\begin{itemize}
	\item[(*)] if $\alpha_*<\wock$, and $\seq{X_s}_{s<\wock}$ is a $\wock$-computable sequence with each $X_s\in (\alpha_*)^\w$ and $X_s\le X_t$ (lexicographically) when $s<t$, then the sequence $\seq{X_s}$ converges to some $X\in (\alpha_*)^\w$. 
\end{itemize}
The proof, as in the argument above that the trees $T_{n,s}$ stabilise, is by admissibility; once $\seq{X_s\rest{n}}$ has stabilised, $\seq{X_s(n)}$ is a non-decreasing sequence of ordinals $<\alpha_*$, and so must stabilise.

\medskip

Thus, when conducting a bad oracle construction, we will choose paths $\seq{\xi_s}$ of strategies with $\xi_s\le \xi_t$ for $s\le t$; as a result, we will be guaranteed that these paths converge to a ``true path'' $\xi = \xi_\wock$. In turn, this will mean that for each~$n$, the sequence of trees $T_{n,s} = T_{\xi_s\rest{n}}$ will converge to $T_n = T_{\xi\rest{n}}$.

\subsection{Computing with the bad oracles} 
\label{subsec:computing_with_the_bad_oracle}

As mentioned, in our construction we will choose $A\in \bigcap_n [T_n]$ (which will usually be unique). We will then sometimes want to show that~$A$ computes some object~$X$. The natural way to do this will be to define a higher Turing functional~$\Psi$ mapping, for each $n$ and~$s$, $\stem(T_{n,s})$ to some finite initial segment, say $X_s\rest{n}$ for example, where $\seq{X_s}$ is an approximation for~$X$. To show that $\Psi(A)=X$, the main task will be showing that $\Psi(A)$ is consistent. In this simple example, this amounts to showing that if $\stem(T_{n,s})\prec A$  then $X_s\rest{n}\prec X$. 

Now the mapping of strategies to trees will usually preserve order and non-order; for strategies~$\alpha$ and~$\beta$, if $\alpha\prec \beta$ then $T_\alpha\prec T_\beta$ but if $\alpha$ and~$\beta$ are incomparable then $T_\alpha$ and~$T_\beta$ will be incomparable in~$\+T$. Recall that this means that $[T_\alpha]$ and $[T_\beta]$ are disjoint, and so if $\alpha\nprec \xi$ then $A\notin [T_\alpha]$. It seems that this is enough to ensure consistency of~$\Psi(A)$: all we need to ensure is that if $\alpha\prec \xi$ then $X_\alpha \prec X$ (where $X_\alpha$ is the initial segment of~$X$ determined by the strategy~$\alpha$). 

However, we need to remember that we are dealing with continuous reductions, which is why we used $\stem(T_\alpha)$ (rather than the paths of $T_\alpha$) as sufficient information to compute $X_\alpha$. While we have $\alpha\preceq \beta$ implying $\stem(T_\alpha)\preceq \stem(T_\beta)$, it is \emph{not} the case that if $\alpha$ and~$\beta$ are incomparable then $\stem(T_\alpha)$ and $\stem(T_\beta)$ are incomparable. This is mostly the case, but $\stem(\Nar(T))= \stem(T)$, which causes complications. 

The argument that~$\Psi$ is consistent will therefore be more delicate. It will rely on the following simple fact:

\begin{fact} \label{fact:small_tool1}
Let $S,T\in \+T$. If $S\succeq T\dbrak{k}$ for some~$k$ then $\stem(S)\notin \Nar(T)$ (and so $\stem(S)\npreceq \s$ for all $\s\in \Nar(T)$). 
\end{fact}

The reason, of course, is that $\s_k(T) = \stem(T\dbrak{k})\preceq \stem(S)$ and $\s_k(T)\notin \Nar(T)$. The argument will also rely on the fact that we enumerate as axioms of~$\Psi$ only pairs $(\stem(T_\alpha), X_\alpha)$ for~$\alpha$ which is accessible at some stage of the construction, rather than all possible strategies~$\alpha$; in particular, we will use the fact that if $\alpha\cat \w$ is never accessible, then we don't map $\stem(T_\alpha) = \stem(T_{\alpha\cat \w})$ to $X_{\alpha\cat \w}$ (which may be longer than $X_\alpha$). We will now show how this is done.

\section{Self-PA oracles} \label{sec:self-PA}

It is not difficult to design a c.e.\ open set $\U \neq 2^\omega$ containing all the computable points; equivalently, to exhibit a {nonempty} $\Pi^0_1$ class with no computable paths. In fact the computational power needed to compute a point in any non-empty effectively closed set is well known and characterized as the PA degrees, i.e., the Turing degrees which can compute a complete and consistent extention of Peano arithmetic. This has many other combinatorial characterizations which easily relativize and even in a uniform way. Namely, there is an effectively closed subset~$\+P$ of the plane such that for all $A\in 2^\w$, the section $\+P^A$ {is nonempty and} does not contain any $A$-computable points, in fact, contains only points which are PA relative to~$A$. 

Our first construction shows that such uniformity cannot be achieved in the higher setting.

\subsection{A warm up example: a bad oracles for an effectively closed set}

\begin{theorem} \label{th:no_A_universal_oracle_continuous_test}
Let $\+P$ be a higher oracle effectively closed set such that for all $B\in 2^\w$, $\+P^B$ is nonempty. Then there is some $A\in 2^\w$ such that $\+P^A$ contains a higher $A$-computable point (a point $X\le_{\wock\Tur}A$). 
\end{theorem}

Further, $A$ can be taken to be higher $\Delta^0_2$. 


\begin{proof}
The construction follows the simple scheme described above assigning strategies $\alpha\in (\w+1)^{<\w}$ to trees $T_\alpha\in \+T$:
\begin{itemize}
	\item $T_{\epsilon} = 2^{<\w}$; 
	\item $T_{\alpha\cat k} = T_\alpha\dbrak{k}$;
	\item $T_{\alpha\cat \w} = \Nar(T_\alpha)$. 
\end{itemize}

We also define a mapping from strategies~$\alpha\in (\w+1)^{<\w}$ to finite binary sequences $X_\alpha$, with $|X_\alpha| = |\alpha|$, by letting, for $n<|\alpha|$, $X_\alpha(n) = 0$ if and only if $\alpha(n)<\w$. 

At stage~$s$ we define the path of strategies $\xi_s$ by recursion on their length. That is, the empty strategy $\epsilon$ is always accessible, and given an accessible $\alpha\prec \xi_s$, we determine an immediate successor $\alpha\cat o\prec \xi_s$. 

Suppose that $\alpha\prec \xi_s$. 
There are two possibilities, depending on whether there is some $k<\w$ such that for all $\s \in T_{\alpha}\dbrak{k}$ we have $[X_\alpha\cat 0]\cap \+P_s^\s \ne \emptyset$. 
\begin{enumerate}[label=(\roman*)]
	\item If there is such, we choose~$k$ to be the least, and declare that $\alpha\cat k\prec \xi_s$. 
	\item If no such~$k$ exists, we declare that $\alpha\cat \w\prec \xi_s$. 
\end{enumerate}

This defines the construction. 


\subsubsection{Convergence} 
We show that if $s<t$ then $\xi_s\le \xi_t$ (in the lexicographic ordering of $(\w+1)^\w$). For let $\alpha$ be a finite common initial segment of $\xi_s$ and~$\xi_t$. For all~$\s$, $\+P^\s_s\supseteq \+P^\s_t$; it follows that if $\alpha\cat k\prec \xi_t$ for some $k<\w$ then for all $\s\in T_\alpha\dbrak{k}$, $\+P^\s_s\cap [X_\alpha]\ne\emptyset$; therefore $\alpha\cat m\prec \xi_s$ for some $m\le k$. On the other hand, if $\alpha\cat \w\prec \xi_s$, then for all~$k$, $\+P^\s_t\cap [X_\alpha] = \emptyset$ for some $\s\in T_\alpha\dbrak{k}$, which implies that $\alpha\cat \w\prec \xi_t$ as well. 

\smallskip

Thus $\seq{\xi_s}$ converges to some $\xi\in (\w+1)^\w$. Now choose any $A\in \bigcap \left\{ T_\alpha \,:\,  \alpha\prec \xi \right\}$. We also let $X = \bigcup_{\alpha\prec \xi} X_\alpha$. 

\subsubsection{Verifying that $X\in \+P^A$} 
By induction on the length of $\alpha\prec \xi$, we show that:
\begin{equation}
[X_\alpha]\cap \+P^B \ne \emptyset \text{ for all } B \in [T_\alpha] \tag{*}
\end{equation}
The base case $|\alpha|=0$ holds by the assumption that $\+P^B$ is nonempty for all~$B$. Suppose that this has been shown for $\alpha$ and let us show it for $\alpha\cat o$, where $\alpha\cat o\prec \xi$ for $o< \w+1$. If $o=k<\omega$ then $X_{\alpha\cat o} = X_\alpha\cat 0$; $T_{\alpha\cat o} = T_\alpha\dbrak{k}$, and for all $B\in [T_{\alpha\cat o}]$, $[X_\alpha\cat 0]\cap \+P^B\ne \emptyset$. This is by compactness: for all $\s\in T_{\alpha\cat o}$, for all~$s$, $[X_\alpha\cat 0]\cap \+P_s^\s\ne\emptyset$. 

Suppose then that $o = \w$; so $X_{\alpha\cat o} = X_\alpha\cat 1$. By construction, for every~$k$ there is $\s\in T_\alpha\dbrak{k}$ such that $[X_\alpha\cat 0]\cap \+P^{\s}= \emptyset$. Suppose, for a contradiction, that there is some $C \in [T_{\alpha\cat o}]= [\Nar(T_{\alpha})]$ such that $[X_{\alpha\cat o}] \cap \+P^C = \emptyset$. By compactness, there is some $\s \in T_{\alpha\cat o}$ such that $[X_\alpha \cat 1]\cap \+P^\s = \emptyset$. There is some $k<\omega$ such that $\s_k(T_{\alpha})$ extends~$\s$, and there is some $B\in [T_{\alpha}\ldbrack k \rdbrack] \subseteq [T_{\alpha}]$ such that $[X_\alpha\cat 0]\cap \+P^{B} = \emptyset$. Since $\s\prec B$ we get $[X_\alpha] \cap \+P^B = \emptyset$, contradicting the induction hypothesis for~$\alpha$. 

Now $X\in \+P^A$ follows from the fact that $\+P^A$ is closed.

\subsubsection{Verifying that $X\le_{\wock\Tur} A$} 
It is immediate that $X\le_{\wock\Tur} \xi$: use the functional $\left\{ (\alpha,X_\alpha) \,:\,  \alpha\in (\w+1)^{<\w} \right\}$.  We show that $\xi\le_{\wock\Tur} A$. 
We define the higher Turing functional
 \[
 \Phi = \left\{ (\stem(T_{\alpha}), \alpha)  \,:\, (\exists s<\wock)\,\alpha\prec \xi_s \right\}.
 \]
 For all $\alpha\prec \xi$ we have $\stem(T_{\alpha})\prec A$ and $\alpha\prec \xi_s$ for some~$s$, so $\alpha\prec \Phi(A)$. So to show that $\Phi(A)=\xi$, it remains to show that~$\Phi$ is consistent on~$A$. That is, to show that for all~$s$, if $\alpha\prec \xi_s$ and $\stem(T_\alpha)\prec A$, then $\alpha\prec \xi$. This is done by induction on the length of~$\alpha$.

 So fix some~$\alpha\in (\w+1)^{<\w}$ and some $s<\wock$, and suppose that $\alpha\prec \xi_s,\xi$, but that $\alpha\cat m\prec \xi_s$ and $\alpha\cat o\prec \xi$, where $m\ne o$. We necessarily have $m<o$, so $m<\w$. Thus $\stem(T_{\alpha\cat m}) = \s_m(T_\alpha)$. We need to show that $\stem(T_{\alpha\cat m})\nprec A$. There are two cases:
 \begin{itemize}
 	\item  If $o<\w$, then $\stem(T_{\alpha\cat o}) = \s_o(T_\alpha)\prec A$, and $\s_m(T_\alpha)$ and $\s_o(T_\alpha)$ are incomparable. 
 	\item If $o=\w$ then $A\in [T_{\alpha\cat o}] = \Nar(T_\alpha)$, while $\s_m(T_\alpha)\notin \Nar(T_\alpha)$ (recall \cref{fact:small_tool1}), so again $\s_m(T_\alpha)\nprec A$. 
 \end{itemize}

\subsubsection{The complexity of~$A$} 
If there are infinitely many~$n$ such that $\xi(n)<\w$, then $\seq{T_n}$ is a shrinking path, so $\bigcap_n[T_n]$ is a singleton, and $\smallseq{\bigcup_{\alpha\prec \xi_s} \stem(T_{\alpha})}_{s<\wock}$ converges to~$A$, so $A$ is higher~$\Delta^0_2$. Otherwise, the leftmost element of $\bigcap_n [T_n]$ is hyperarithmetic. The latter option is possible, for example consider when $\+P^B = \{1^\infty\}$ for all~$B$. 
\end{proof}

\subsection{Self-PA oracles} 
\label{sub:self_pa_oracles}

\Cref{th:no_A_universal_oracle_continuous_test} shows that the uniform construction of uniform PA-complete effectively closed sets fails in the higher setting. One can wonder what happens if we drop the uniformity requirement. 

\begin{definition} \label{def:self-PA_oracles}
	An oracle $A$ is \emph{higher self-$\PA$} if every nonempty higher $A$-co-c.e.\ closed set contains some higher $A$-computable point.  
\end{definition}

In this section we show that for computing points in closed sets, the situation in the higher setting does not deviate too much from the standard one:

\begin{theorem} \label{thm:no_self-PA}
	There is no higher self-$\PA$~oracle. 
\end{theorem}

In contrast, Miller and Soskova showed that there are self-PA enumeration degrees. 

\medskip

The proof of \cref{thm:no_self-PA}, necessarily non-uniform by \cref{th:no_A_universal_oracle_continuous_test}, splits into two cases, which we treat separately, for the first stating a slightly stronger result:

\begin{proposition} \label{prop:no-self-PA:incomplete_case}
	There is {a} higher oracle effectively closed set~$\+P$ such that:
	\begin{enumerate}[label=\textup{(\alph*)}]
		\item For all $B\in 2^\w$, $\+P^B\ne \emptyset$; and
		\item If $A \nge_{\wock\Tur} O$ then $\+P^A$ contains no higher $A$-computable points.
	\end{enumerate}
\end{proposition}

\begin{proof}
	Let $p\colon \wock\to \w$ be a $\wock$-computable (injective) enumeration of~$O$. 

	\smallskip
	
	We define a higher oracle effectively closed set~$\+P$ as follows. At every stage $s<\wock$, for every $e<\w$, for every string $\tau\in 2^{<\w}$, minimal with respect to the property:
	 \[
	 \Phi_{e,s}(\tau) \text{ is consistent, and } |\Phi_{e,s}(\tau)| > p(s)^2 + e^2
	 \]
	remove $[\Phi_{e,s}(\tau)]$ from $\+P^\tau$. 

	\medskip
	
	First we observe that for all~$B$, $\+P^B$ is nonempty. Indeed, for each~$n$, fewer than~$n$ many strings of length $\le n$ are ever removed from $\+P^B$. For suppose that $[\s]$ is removed from $\+P^B$, and $|\s|\le n$. Then $\s = \Phi_{e,s}(\tau)$ for some $\tau\prec B$ and is removed at some stage~$s$; note that at each stage we act for at most one $\tau\prec B$. We have $|\s| > p(s)^2+e^2$ so $p(s),e< \sqrt{n}$; as~$p$ is injective, there are fewer than~$n$ many such pairs $(e,s)$. 

	\medskip
	
	Next, let $A\in 2^\w$, $e<\w$, and suppose that $\Phi_e(A)$ is total, consistent, and an element of $\+P^A$. Suppose that $\tau\prec A$, $s<\wock$ and $|\Phi_{e,s}(\tau)|> n^2+e^2$; then $O_s\rest{n} = O\rest{n}$: otherwise there is some $t>s$ with $p(t)<n$, and as $|\Phi_{e,t}(\tau)|\ge |\Phi_{e,s}(\tau)|$, we would remove $\Phi_e(A)$ from $\+P^A$ at stage~$t$. So the functional 
	\[
		\Psi = \left\{ (\tau, O_s\rest{n}) \,:\, s<\wock, \Phi_{e,s}(\tau) \text{ is consistent, and } |\Phi_{e,s}(\tau)|> n^2+e^2  \right\}
	\]
	gives $\Psi(A)=O$. 
\end{proof}

\begin{proof}[Proof of \cref{thm:no_self-PA}]
	In light of \cref{prop:no-self-PA:incomplete_case}, it remains to show that if $A\ge_{\wock\Tur}O$ then there is a higher $A$-effectively closed set containing no higher $A$-computable points. 

	\smallskip
	
	Let 
	\[
		C = \left\{ (e,\s) \,:\,  \Phi_e(\s)\text{ is consistent} \right\}.
	\]
	Then $C\le_\Tur O$. Let~$\Psi$ be a higher functional such that $\Psi(A)= C$. We define a higher oracle effectively closed set $\+P$: for any~$\tau$ and~$e$, if there is some $\s\preceq \tau$ such that $\Psi(\tau,e,\s) = \yes$,\footnote{Meaning that~$\tau$ thinks that $(e,\s)\in C$.} and some $s<\wock$ such that $\Phi_{e,s}(\s)$ is consistent and has length at least $e+1$, then we remove $[\Phi_{e,s}(\s)]$ from $\+P^\tau$. 

	\smallskip
	
	If $\rho = \Phi_{e,s}(\s)$ is removed from $\+P^A$ then $(e,\s)\in C$. It follows that if $\Phi_{e,s}(\s)$ and $\Phi_{e,s'}(\s')$ are removed from $\+P^A$ then these removed strings are comparable; the losses on behalf of~$e$ thus amount to no more than a single string of length at least~$e+1$. It follows that $\+P^A$ is nonempty. 

	If $\Phi_e(A)$ is total and consistent then for sufficiently long $\s\prec A$ and sufficiently late~$s$ we have $|\Phi_{e,s}(\s)|>e$; and for sufficiently long $\tau\prec A$ we have $\Psi(\tau,e,\s) = \yes$; so $\Phi_e(A)\notin \+P^A$.  

	\smallskip
	
	Note that it is quite possible that if $\Psi(B)\ne C$ then $\+P^B$ is empty.
\end{proof}

\subsection{Bad oracles for uniform self-PA} 
\label{sub:bad_uniform_self_pa}

In light of \cref{thm:no_self-PA}, the best we can hope for is badness for uniform failure of self-PA-ness.

\begin{definition} \label{def:bad_for_uniform_self-PA}
	An oracle~$A$ is \emph{bad for uniform self-$\PA$} if for every higher oracle effectively closed set~$\+P$, if for every~$B$, $\+P^B\ne \emptyset$, then $\+P^A$ contains some higher $A$-computable point. 
\end{definition}

\begin{theorem} \label{thm:existence_of_bad_for_uniform_higher_self_PA}
	There is an oracle which is bad for uniform self-$\PA$. 
\end{theorem}

In fact, there is a higher $\Delta^0_2$ such oracle.

\begin{proof}
	Not much needs to be added to the construction of \cref{th:no_A_universal_oracle_continuous_test}. We tackle all higher oracle effectively closed sets and interleave the constructions. 

	Fix a computable pairing function $(e,d)\mapsto \seq{e,d}$ satisfying $\seq{e,d}< \seq{e,d+1}$ for all~$e$ and~$d$; for example, the standard Cantor pairing function will do. Let $\seq{\+P_e}$ be an effective list of all higher oracle effectively closed sets. We again let the collection of srategies be $(\w+1)^{<\w}$, with the same scheme $\alpha\mapsto T_\alpha$. We now approximate sets $X_e$ for all $e<\w$, with the intention that $X_e\in \+P_e^A$ is higher $A$-computable unless $\+P_e^B$ is empty for some~$B$. For every strategy~$\alpha$ and every $e<\w$ we define $X_{e,\alpha}$ by induction on $|\alpha|$: 
	\begin{itemize}
		\item For every~$e$, $X_{e,\epsilon} = \epsilon$; 
		\item If $|\alpha|= \seq{e,d}$ then $X_{e,\alpha\cat k} = X_{e,\alpha}\cat 0$ for $k<\w$, and $X_{e,\alpha\cat \w} = X_{e,\alpha}\cat 1$. For all $e'\ne e$, $X_{e',\alpha\cat o} = X_{e',\alpha}$ for all $o\le \w$. 
	\end{itemize}
	So $X_{e,\alpha}(d)$ is defined if and only if $\seq{e,d}<|\alpha|$. 

	The construction is as expected: suppose that $\alpha\prec \xi_s$; let $|\alpha| = \seq{e,d}$. If there is some~$k$ such that for all $\s\in T_\alpha\dbrak{k}$, $[X_{e,\alpha}\cat 0]\cap \+P^\s_{e,s}\ne \emptyset$, then we let $\alpha\cat k\prec \xi_s$ for the least such~$k$; otherwise we let $\alpha\cat \w\prec \xi_s$.

	For the verification, by induction on $\alpha\prec \xi$ we prove that for all~$e$, \emph{if} for all $B$, $\+P^B_e\ne\emptyset$, then $[X_{e,\alpha}]\cap \+P^B_e\ne\emptyset$ for all $B\in [T_\alpha]$. The proof that $\xi\le_{\wock\Tur} A$ is the same. 

	Note that if $e$ is such that for all~$B$, $\+P^B = 2^\w$, then for all~$d$, $\xi(\seq{e,d}) = 0$. It follows that $\seq{T_{\alpha}}_{\alpha\prec \xi}$ is a shrinking path, i.e., $\bigcap_{\alpha\prec \xi} [T_\alpha]$ is a singleton (and the same holds at every $s<\wock$). 
\end{proof}

Note that \cref{prop:no-self-PA:incomplete_case} implies that any oracle which is bad for uniform  self-PA must higher compute~$O$.

\section{Higher Turing consistent computations} \label{section:separation_fin_h_higher_turing}

\subsection{A bad oracle for consistent functionals}

The following result justifies the use in \cite{BGMContinuousHigherRandomness} of the general form of higher Turing reductions, rather than everywhere consistent reductions.

\begin{theorem} \label{thm:bad_oracle_for_consistent_Turing}
	There is an oracle~$A$ and some $X\le_{\wock\Tur} A$ such that for no everywhere consistent higher functional~$\Phi$ do we have $\Phi(A)=X$. 
\end{theorem}

Again we can make~$A$ higher~$\Delta^0_2$. We give a direct construction first. 

\begin{proof}
	The construction is similar to that of \cref{th:no_A_universal_oracle_continuous_test}. Note that we have an effective enumeration of all consistent higher functionals $\seq{\Phi_e}$. We can copy the axioms of a functional and stop right before we see an inconsistency creep in. 

	We use the same scheme for mapping strategies to trees as the two constructions above; and the same mapping $\alpha\mapsto X_\alpha$ as in the proof of \cref{th:no_A_universal_oracle_continuous_test}. 

	At stage~$s$, if we have already determined that $\alpha\prec \xi_s$, and $|\alpha|=e$, then we let $\alpha\cat k\prec \xi_s$ if~$k$ is least such that for no $\tau\in T_\alpha\dbrak{k}$ do we have $\Phi_{e,s}(\tau)\succeq X_\alpha\cat 0$. If there is no such~$k$ we let $\alpha\cat \w\prec \xi_s$. 

	Note that if $\Phi_e$ is the empty functional, then $\alpha\cat 0\prec \xi$, so $\seq{T_{\alpha}}_{\alpha\prec \xi}$ is a shrinking path; as usual let $\{A\} = \bigcap_{\alpha\prec \xi}[T_\alpha]$. We claim that for all $e<\w$, $\Phi_e(A)\ne X$: let $\alpha\prec \xi$ have length~$e$. If $\alpha\cat k\prec \xi$, then $\Phi_e(A)$ does not extend $X_\alpha\cat 0 = X\rest{e+1}$. Otherwise, $A\in \Nar(T_\alpha)$ and densely along~$A$, $\Phi_e$ maps strings to $X_\alpha\cat 0$ which is incomparable with~$X$, so $\Phi_e$ cannot map any initial segment of~$A$ to $X\rest{e+1}$. The proof that $\xi\le_{\wock\Tur} A$ is the same. 
\end{proof}

The construction above is almost identical to the one given for \cref{th:no_A_universal_oracle_continuous_test}; in fact we can derive the result from that construction, by applying it to the higher oracle effectively closed set given by the following proposition.

\begin{proposition} \label{prop:avoiding_consistently_computable_reals}
	There is a higher oracle effectively closed set~$\+P$ such that for all~$B$, $\+P^B$ is nonempty but contains no $X$ which is higher $B$-computable by a consistent functional. 
\end{proposition}

\begin{proof}
	Let $\seq{\Phi_e}$ list all consistent higher functionals. For each~$\tau$ and~$e$, if $|\Phi_e(\tau)|>e$, remove $\Phi_e(\tau)$ from~$\+P^\tau$. 
\end{proof}

An oracle satisfying \cref{thm:bad_oracle_for_consistent_Turing} is called \emph{bad for consistent functionals}. \Cref{prop:avoiding_consistently_computable_reals} implies that an oracle which is bad for uniform self-PA is also bad for consistent functionals. It is not clear if these notions are equivalent; it seems possible that the former is stronger, because it deals with all higher oracle co-c.e.\ closed sets, rather than just the one built in \cref{prop:avoiding_consistently_computable_reals}. We do not even know whether a bad oracle for consistent functionals must higher compute~$O$.\footnote{Note that a higher~$\Delta^0_2$ oracle which is bad for consistent Turing functionals cannot compute~$O$ via a consistent functional; this is because~$X$ is computable from~$O$ via a consistent functional, being higher~$\Delta^0_2$.} Note that we do know that if~$A$ is bad for consistent functionals then $\w_1^A>\wock$ \cite[Prop.2.3]{BGMContinuousHigherRandomness}. 

We now introduce another ``highness property'' about which we can say something, and will be useful in what comes next.

\subsection{Tree-collapsing oracles} 
\label{sub:tree_collapsing_oracles}

An oracle $A$ \emph{collapses} $\wock$ if $\w_1^A> \wock$, equivalently, if $O\le_h A$. In \cite{BGMContinuousHigherRandomness} we examined a variety of stronger properties of oracles within the higher $\Delta^0_2$ sets, again based on properties of their approxmations. In particular, we used the notion of a collapsing approximation.

Let $\{X_s\}_{s < \wock}$ be a $\wock$-computable approximation of a set~$X$. For $n<\w$ let $s(n)$ be the least stage $s<\wock$ such that $X_s\rest n  = X\rest n$. The approximation $\seq{X_s}$ is \emph{collapsing} if $\sup_n s(n) = \wock$. Equivalently, if for all $s<\wock$, $X$ does not belong to the closure of the set $\{X_t\,:\, t<s\}$. 

\begin{definition} \label{def:tree-collapsing}
	An oracle~$A$ is \emph{tree-collapsing} if there is a $\Pi^1_1$ tree $T\subseteq 2^{<\w}$ such that $A\in [T]$ but for every $\Delta^1_1$ subtree $S\subseteq T$, $A\notin [S]$. 
\end{definition}

If $\seq{T_s}$ is a $\wock$-computable enumeration of a~$\Pi^1_1$ tree~$T$, then we may assume that each~$T_s$ is a tree. An admissibility argument shows that $\bigcup_s [T_s]$ is the collection of all $A\in [T]$ which lie on some $\Delta^1_1$ subtree of~$T$. Thus~$T$ witnesses that~$A$ is tree-collapsing if and only if letting~$s(n)$ be the least~$s$ such that $A\rest{n}\in T_s$, we have $\sup_n s(n)= \wock$. This shows that if~$A$ is tree-collapsing then~$A$ collapses~$\wock$.

\begin{proposition} \label{prop:tree-collapsing_for_Delta_2}
	If $A$ is higher $\Delta^0_2$, then~$A$ is tree-collapsing if and only if it has a collapsing approximation. 
\end{proposition}

\begin{proof}
	If $\seq{A_s}$ is a collapsing approximation of~$A$, we let 
	 \[
	 T = \left\{ A_s\rest{n} \,:\,  n<\w\andd s<\wock \right\}.
	 \]
	 In the other direction, let $\seq{A_s}$ be a $\wock$-computable approximation of~$A$ and let~$T$ witness that~$A$ is tree-collapsing. For each~$s$ let $\s_s\in 2^{\le \w}$ be the longest initial segment of~$A_s$ on~$T_s$. Let $B_s = \s_s$ if $\s_s$ is infinite, otherwise let $B_s = \s_s\conc 0^\infty$. Then $\seq{B_s}$ is a collapsing approximation of~$A$ (using the fact that~$A$ is not computable, so does not end with a string of zeros). 
\end{proof}

In some sense, tree-collapsing oracles are better than having no oracle at all: they allow us to replace processes of length~$\wock$ by processes of length~$\w$, thereby allowing us to revert to some time tricks. Compare the following to \cref{th:prefix_free_failure}. 

\begin{proposition} \label{prop:tree_collapsing_and_Pi11_antichains}
	If~$A$ is tree collapsible then every higher $A$-c.e.\ open set is generated by a higher $A$-computable prefix-free set. 
\end{proposition}

\begin{proof}
	Let~$T$ witness that~$A$ is tree-collapsing, and let $W$ generate an oracle c.e.\ open set. We assume that for all pairs $(\tau,\s)\in W$ we have $|\tau|\ge |\s|$. 

	Define an enumeration $\seq{V_s}_{s<\wock}$ as follows. Start with $V_0 = \emptyset$, and at limit stages take unions. Suppose that~$V_s$ has been defined. We enumerate a new pair $(\tau,\s)$ into $V_{s+1}$ if $\tau\in T_{s+1}\setminus T_s$, $|\s|= |\tau|$, $[\s]\subseteq \+W_s^\tau$ and $[\s]\cap \+V_s^\tau = \emptyset$. 

	For all~$B$ we have $\+V^B \subseteq \+W^B$, and $V^B$ is a higher $B$-computable prefix-free set. On the other hand, if $[\s]\subseteq \+W^A$ then by compactness, there is some $\tau\prec A$ and some $s<\wock$ such that $[\s]\subseteq \+W_s^\tau$. There is some $\tau'\prec A$, extending~$\tau$, which is enumerated into $T_{t+1}$ for some $t\ge s$. Then $[\s]\subseteq \+V^{\tau'}_{t+1}$, which shows that $\+W^A = \+V^A$. 
\end{proof}


An oracle is good if it is not bad. 

\begin{proposition} \label{prop:tree_collapsing_and_good_for_consistent_functionals}
	If $A$ is tree-collapsing then it is good for consistent functionals. 
\end{proposition}

\begin{proof}
	Similar to the proof of \cite[Prop.2.2]{BGMContinuousHigherRandomness}, which proves the same for oracles with collapsing approximations. Letting $\Phi$ be a higher Turing functional, we copy some $\Phi$-computations into a consistent functional~$\Psi$; if $\tau\in T_{s+1}\setminus T_s$, and $\s = \Phi_s(\tau)$ is consistent, then we add the axiom $(\tau,\s)$ to $\Psi_{s+1}$. 
\end{proof}

We can pay a debt made in \cite{BGMContinuousHigherRandomness}:

\begin{corollary}
There exists a higher $\Delta^0_2$ oracle $A$ such that $\omega_1^A > \wock$ and such that~$A$ does not have a collapsing approximation.
\end{corollary}
\begin{proof}
Take~$A$ to be higher $\Delta^0_2$ and bad for consistent functionals; by \cref{prop:tree-collapsing_for_Delta_2,prop:tree_collapsing_and_good_for_consistent_functionals}, $A$ does not have a collapsing approximation, but we know that $\w_1^A>\wock$. 
\end{proof}

\subsection{fin-h reduction} 
\label{sub:fin_h_reduction}

The first notion of higher continuous reductions was introduced by Hjorth and Nies in \cite{HjorthNies2007}. They defined a \emph{fin-h functional} to be a higher functional~$\Phi$ which (a) as a set of pairs, is a monotone function from strings to strings; and (b) $\dom \Phi$ is a subtree of~$2^{<\w}$. This means that for all~$\tau$, when we see $\tau\in \dom \Phi$, we will have seen all proper initial segments of~$\tau$ in~$\dom \Phi$ as well, and $\Phi(\tau)$ is fixed from that time onwards. Hjorth and Nies write $Y\le_{\finh} X$ if $\Phi(X)=Y$ for some fin-h functional~$\Phi$. 

Certainly every fin-h functional is consistent, but a consistent functional may allow axioms to be added in reverse order some times. We may expect then that a bad oracle construction could separate fin-h from general reductions by consistent functionals. This is not so. 

\begin{theorem}
Suppose that $Y = \Phi(X)$ for a consistent functional~$\Phi$. Then $Y\le_{\finh} X$. 
\end{theorem}

The proof is non-uniform, and this non-uniformity is necessary; it is not difficult to meet a single requirement when trying to separate between consistent and fin-h reductions. 

\begin{proof}
If $X$ is tree-collapsing, then the construction of \cref{prop:tree_collapsing_and_good_for_consistent_functionals} actually gives a fin-h functional. 

\smallskip

Suppose that~$X$ is not tree-collapsing. For $s<\wock$ let 
\[
	T_s = \left\{ \s\in 2^{<\w} \,:\,  (\forall n<\w)(\exists \tau\succeq \s)\,\,|\Phi_s(\tau)|\ge n \right\}
\]
and let $T = \bigcup_s T_s$. We first note that $X\in [T]$; for all $\s\prec X$, for all~$n$, we know that there is some $\tau\succeq \s$ with $|\Phi(\tau)|\ge n$; by admissibility, for each~$n$ such~$\tau$ appears by some bounded stage, i.e., $\s\in T_s$ for some~$s$. The assumption that~$X$ is not tree-collapsing means that $X\in [T_s]$ for some~$s$; fix such~$s$. 

We define a fin-h functional $\Psi$ with $\dom \Psi = T_s$. For $\s\in T_{s}$ we let~$\Psi(\s)$ be the longest string~$\rho$ with $|\rho|\le |\s|$ and~$\rho$ comparable with $\Phi_s(\tau)$ for all $\tau\succeq \s$. Because $\s\in T_s$, there is some $\tau\succeq \s$ with $|\Phi_s(\tau)|\ge |\s|$, so we cannot have two incomparable strings~$\rho$ which are candidates for~$\Psi(\s)$. It is not difficult to see that~$\Psi$ is a fin-h functional. We argue that $\Psi(X)=\Phi(X)$. 

We show that $\Psi(X)\preceq \Phi(X)$.  For $\s\prec X$ find $\tau\succeq \s$ with $|\Phi_s(\tau)|\ge |\s|$, so $\Psi(\s)\preceq \Phi_s(\tau)$. Since~$\Phi$ is consistent, $\Phi(\s)$ is comparable with $\Phi_s(\tau)$; so $\Psi(\s)$ is comparable with $\Phi(\s)$. 

It remains to show that $\Psi(X)$ is total. Let $n<\w$; find $\s\prec X$, $|\s|\ge n$, such that $|\Phi(\s)|\ge n$. For all $\tau\succeq \s$, $\Phi(\s)$ is comparable with $\Phi_s(\tau)$; so $|\Psi(\s)|\ge n$. 
%
\end{proof}

\begin{remark}
	The proof above shows that if~$X$ is not tree-collapsible and $Y\le_{\finh} X$ then there is some hyperarithmetic~$H$ such that $Y\le_{\Tur} X\oplus H$. 
\end{remark}

\subsection{Bad random oracles for Turing functionals} \label{sec:bad_random}

We now investigate properties of bad oracles in a different direction: we show that an oracle bad for consistent functionals can be higher Martin-L\"{o}f-random. Recall that a \emph{higher $\ML$-test} is a sequence $\seq{\+U_n}$ of uniformly $\Pi^1_1$-open sets with $\leb(\+U_n)\le 2^{-n}$; here $\leb$ denotes the fair-coin measure on Cantor space. The intersection $\bigcap_n \+U_n$ is called \emph{higher $\ML$-null}, and a sequence $X\in 2^\w$ is \emph{higher $\ML$-random} if it is not an element of any higher $\ML$-null set. In this subsection we prove:

\begin{theorem} \label{th:badmlrandomoracle}
There is a higher $\ML$-random oracle~$A$ which is bad for consistent functionals. 
\end{theorem}
Further, we can make~$A$ higher $\Delta^0_2$.

\subsubsection{Uniform splitting levels}

There is a nonempty $\Sigma^1_1$ closed set consisting only of higher ML-random sequences. This set is the collection of paths through a $\Sigma^1_1$ tree $T\subset 2^{<\w}$; this tree has a co-enumeration: a $\wock$-computable sequence $\seq{T_s}_{s<\wock}$ of trees such that $T = \bigcap_s T_s$. 

Since~$T$ does not contain a hyperarithmetic path, it is perfect. The general idea of the construction is to mimic our previous constructions of bad oracles, but instead of starting with $2^{<\w}$, to start with~$T$. Thus we will approximate during the construction both the path of strategies and the initial tree, which at stage~$s$ will be~$T_s$. 

However, we cannot naively just perform the same construction with the only modification being $T_{\epsilon,s} = T_s$. This is because under the standard definitions, our operations are not monotone. For example, there may be $s<t$, so $T_s\supseteq T_t$, but $\Nar(T_s)\nsupseteq \Nar(T_t)$: this is because the disappearance of nodes from~$T_s$ may cause branching nodes of odd level to become branching nodes of even level. This would completely mess up the construction. 

However, with a little cost, we can avoid this problem by making sure that we only take splittings at uniform levels across all trees. This is done as follows. 

\medskip

For a measurable set~$\+A$ and a non-null measurable set~$\+B$ we let $\leb(\+A \given \+B)$ denote the conditional measure $\leb(\+A\cap \+B)/\leb(\+B)$. For a string $\s\in 2^{<\w}$ we write $\leb(\+A\given \s)$ for $\leb(\+A \given [\s])$. The following is well-known, and follows from additivity of measure.

\begin{lemma} \label{lemma:easy}
Let $\s$ be a string and $\+A$ a measurable set such that $\lambda(\+A\given  \s) \geq 2^{-n}$. Then there are at least two extensions $\tau_1,\tau_2$ of $\s$ of length $|\s|+n+1$ so that for both $i \in \{1,2\}$ we have $\leb(\+A\given \tau_i) \geq 2^{- n-1}$.
\end{lemma}

Now let $f(1) = 0$ and $f(n+1) = f(n)+n+1$. Let $\+A\subseteq 2^\w$ with $\leb(\+A)\ge 1/2$. Then for all~$n\ge 1$, every string~$\tau$ of length $f(n)$ for which $\leb(\+A\given \tau)\ge 2^{-n}$ has at least two extensions~$\s$ of length~$f(n+1)$ for which $\leb(\+A\given \s)\ge 2^{-n-1}$. 

\smallskip

Consider again the higher co-c.e.\ tree~$T$ whose paths are all higher ML-random; as mentioned, it has a co-enumeration $\seq{T_s}$, with $s<t$ implying $T_s\supseteq T_t$, and each~$T_s$ a tree. We may assume that $\leb([T])\ge 1/2$. By further removing strings, we may assume that for every~$s$, every $\s\in T_s$ of length~$f(n)$ has at least two extensions on~$T_s$ of length $f(n+1)$. We simply remove strings~$\s$ of length~$f(n)$ for which $\leb([T_s]\given \s)< 2^{-n}$.

\smallskip

Our plan now is to use the levels~$f(n)$ as locations for splittings. The price to pay, we shall see, is that the splittings are now no longer binary; a string on~$T_s$ of length~$f(n)$ may have more than two extensions on~$T_s$ of length $f(n+1)$. This will require some light modifications to the construction and the verification of convergence. For now, we define the new version of narrow subtrees; these will have an extra parameter, for which we will use the length of the associating strategy. Let $\seq{A_k}_{k<\w}$ be a computable partition of~$\w$ into infinite sets. For $k<\w$ and a tree $R\subseteq 2^{<\w}$ we let 
\[
	\Nar_k(R)
\]
be the tree obtained from~$R$ by removing from~$R$, for each $n\in A_k$ such that $f(n)> |\stem(R)|$ and each $\s\in R$ of length $f(n)$, the \emph{rightmost} extension of~$\s$ on~$R$ of length $f(n+1)$.

We will use the following:

\begin{fact} \label{fact:Nar_k_subtree} \
\begin{enumerate}[label=(\alph*)]
	\item If $R\subseteq S$ then for all~$k$, $\Nar_k(R)\subseteq \Nar_k(S)$. 
	\item If $R = \bigcap_{s<\wock} R_s$ then $\Nar_k(R) = \bigcap_{s<\wock} \Nar_k(R_s)$. 
\end{enumerate}
\end{fact}

For (b), the point is that every $\tau$ of length~$f(n)$ has only finitely many extensions of length $f(n+1)$, so the rightmost such extension eventually stabilises. 

\subsubsection{The moving tree of trees}

Let $\seq{\tau_k}_{k<\w}$ be a computable enumeration of all finite binary strings which have length $f(n)$ for some~$n$. Ensure that $k\le m$ implies $|\tau_k|\le |\tau_m|$. For each~$n$, for each string~$\tau$ of length $f(n)$, let $\s_0(\tau), \s_1(\tau), \dots, $ be an enumeration of all extensions of~$\tau$ of length $f(n+1)$, enumerated \emph{from right to left}. 

Our collection of strategies is a sub-collection of $(\w^2+1)^{<\w}$. The outcomes of a strategy~$\alpha$ are:
\begin{itemize}
	\item For each $k$ and each~$i$ such that $\s_i(\tau_k)$ is defined, the outcome $\w k+i$; 
	\item The outcome $\w^2$. 
\end{itemize}
Note that $\w k+i$ is an outcome exactly for $i<2^{f(n+1)-f(n)}$ (where $|\tau_k| = f(n)$). Thus the order-type of all outcomes is $\w+1$. For clarity, we choose not to renumber them. 

In what follows, let $T_{\wock} = T = \bigcap_{s<\wock} T_s$. For a strategy~$\alpha$ and a stage $s\le\wock$ we define a tree $T_{\alpha,s}$. We do this by recursion on the length of~$\alpha$.
\begin{itemize}
	\item $T_{\epsilon,s} = T_s$. 
	\item Given $T_{\alpha,s}$, we let $T_{\alpha\cat (\w k+i),s} = T_{\alpha,s}\rest{\s_i(\tau_k)}$ (the full subtree issuing from the string $\s_i(\tau_k)$);
	\item and we let $T_{\alpha\cat\w^2,s} = \Nar_{|\alpha|} (T_{\alpha,s})$.
\end{itemize}

Note that it is possible for $T_{\alpha,s}$ to be finite; for example if $\alpha\succeq \beta\cat(\w k+i)$ where $\s_i(\tau_k)\notin T_{\beta,s}$. We will make sure that such a strategy~$\alpha$ will never be accessible at stage~$s$. 

The following is not difficult to verify.

\begin{lemma} \label{lem:moving_tree_of_trees:basic_facts} 
Let $\alpha$ be a strategy and let $s\le \wock$. 
	\begin{enumerate}[label = (\alph*)]
		\item For all $t\ge s$, $T_{\alpha,s}\supseteq T_{\alpha,t}$. 
		\item $T_{\alpha,\wock} = \bigcap_{t<\wock} T_{\alpha,t}$. 
		\item For all $\beta \succeq \alpha$, $T_{\alpha,s}\supseteq T_{\beta,s}$. 
		\item Suppose that $T_{\alpha,s}$ is infinite. Let $n\ge 1$ and let $\s\in T_{\alpha,s}$ have length $f(n)$. 
		\begin{enumerate}[label=(\roman*)]
			\item $\s$ has an extension on~$T_{\alpha,s}$ of length $f(n+1)$. 
			\item If $f(n)>|\stem(T_{\alpha,s})|$ and $n\notin \bigcup_{k<|\alpha|}A_k$ then~$\s$ has at least two extensions on $T_{\alpha,s}$ of length $f(n+1)$. 
		\end{enumerate}
		\item If $T_{\alpha,s}$ is infinite, then so is $\Nar_{|\alpha|}(T_{\alpha,s})$. 
	\end{enumerate}
\end{lemma}

Also, for a strategy $\alpha$ we define $X_\alpha$. We declare that $|X_\alpha|= |\alpha|$; $X_\alpha(n) = 0$ if $\alpha(n)<\w^2$, otherwise $X_\alpha(n) = 1$.

\subsubsection{Bad random oracle}

We now use our settings to give the desired proof.

\begin{proof}[Proof of \cref{th:badmlrandomoracle}]
Let $\+P$ be the higher oracle effectively closed set given by \cref{prop:avoiding_consistently_computable_reals}: for every~$A$, $\+P^A$ is nonempty and contains no set~$X$ higher reducible to~$A$ by a consistent functional. So our aim is to find some $A\in [T]$  and some $X\le_{\wock\Tur} A$ such that $X\in \+P^A$. Let $\seq{\+P_s}$ be a co-enumeration of $\+P$; we let $\+P_{\wock} = \+P$. 

\medskip

Let $s\le \wock$. At stage~$s$ we define the path~$\xi_s$ of strategies accessible at stage~$s$. Of course $\epsilon\prec \xi_s$. Suppose that we have already determined that $\alpha\prec \xi_s$. Let 
\[
	B_{\alpha,s} = \left\{ \s\in 2^{<\w} \,:\,  (\exists \tau\succeq \s)\,\,\+P^\tau_s\cap [X_\alpha\cat 0] = \emptyset   \right\}. 
\]
We choose as follows. 
\begin{enumerate}
	\item If $\Nar_{|\alpha|}(T_{\alpha,s})\subseteq B_{\alpha,s}$ then we let $\alpha\cat \w^2\prec \xi_s$. 
	\item Otherwise, let~$k$ be least such that:
	\begin{enumerate}[label=(\roman*)]
		\item $\tau_k\in \Nar_{|\alpha|}(T_{\alpha,s})\setminus B_{\alpha,s}$;
		\item $|\tau_k| = f(n) > |\stem(T_{\alpha,s})|$ and $n\in A_{|\alpha|}$. 
	\end{enumerate}
	Let $i$ be such that $\s_i(\tau_k)$ is the rightmost extension of $\tau_k$ on $T_{\alpha,s}$ of length $f(n+1)$. We let $\alpha\cat (\w k+i)\prec \xi_s$. 
\end{enumerate}

This concludes the construction. Note that we have defined $\xi_s$ for $s=\wock$ as well (starting with $T_{\wock} = T$, and using $\+P_{\wock} = \+P$). This is of course not part of the $\wock$-computable construction; only the sequence $\seq{\xi_s}_{s<\wock}$ is $\wock$-computable. We will show that $\seq{\xi_s}$ converges to $\xi_{\wock}$. We now start the verification.

\subsubsection{Trees are infinite}
For all $s\le \wock$ and $\alpha\prec \xi_s$, $T_{\alpha,s}$ is infinite: this is because when we choose $\alpha\cat (\w k+i)\prec \xi_s$ then $\s_i(\tau_k)\in T_{\alpha,s}$.  This shows that the instruction in case~(2) can be carried out. Note that if $\alpha\cat (\w k+i)\prec \xi_s$ then $\s_i(\tau_k)\in T_{\alpha,s}\setminus \Nar_{|\alpha|}(T_{\alpha,s})$.







\subsubsection{Convergence} We show that for all $s\le t \le \wock$, $\xi_s \le \xi_t$. Suppose that $\alpha\prec \xi_s$ and $\alpha\prec \xi_t$. 

First suppose that $\alpha\cat \w^2\prec \xi_s$. Since $B_{\alpha,s}\subseteq B_{\alpha,t}$ and $\Nar_{|\alpha|}(T_{\alpha,s})\supseteq \Nar_{|\alpha|}(T_{\alpha,t})$, we get $\alpha\cat \w^2\prec \xi_t$ as well. 

Next suppose that $\alpha\cat (\w k+i)\prec \xi_t$. The inclusions in the previous case show that $\tau_k\in \Nar_{|\alpha|}(T_{\alpha,s})\setminus B_{\alpha,s}$ as well; and $|\tau_k|> |\stem(T_{\alpha,t})| \ge |\stem(T_{\alpha,s})|$ (again because $T_{\alpha,s}\supseteq T_{\alpha,t}$). Also, $\s_i(\tau_k)\in T_{\alpha,s}$. Hence $\alpha\cat (\w k'+i')\prec \xi_s$ for some $k'\le k$, and if $k'=k$, for $i'\le i$; this means $\w k'+i'\le \w k+i$. 

\smallskip

It follows that $\seq{\xi_s}_{s<\wock}$ converges to some $\xi \le\xi_{\wock}$. However, $\xi = \xi_{\wock}$. To see this suppose that $\alpha\prec \xi_{\wock}, \xi$. If $\alpha\cat \w^2\prec \xi_{\wock}$, that is, if $\Nar_{|\alpha|}(T_{\alpha,\wock})\subseteq B_{\alpha,\wock}$, then as $B_{\alpha,\wock} = \bigcup_{s<\wock} B_{\alpha,s}$ and $\Nar_{|\alpha|}(T_{\alpha,\wock}) = \bigcap_{s<\wock} \Nar_{|\alpha|}(T_{\alpha,s})$, by admissibility of~$\wock$, there is some~$s$ such that $\Nar_{|\alpha|}(T_{\alpha,s})\subseteq B_{\alpha,s}$. Hence $\alpha\cat \w^2\prec \xi$. 

Suppose that $\alpha\cat (\w k+i)\prec \xi$. So for all $s<\wock$, $\tau_k \in \Nar_{|\alpha|}(T_{\alpha,s})\setminus B_{\alpha,s}$; so $\tau_k\in \Nar_{|\alpha|}(T_{\alpha,\wock})\setminus B_{\alpha,\wock}$. Also, for all~$s$, $|\tau_k| > |\stem(T_{\alpha,s})|$ and as $\stem(T_{\alpha,\wock}) = \lim_{s\to \wock} \stem(T_{\alpha,s})$, we get $|\tau_k|> |\stem(T_{\alpha,\wock})|$. A similar argument shows that $\s_i(\tau_k)$ is the rightmost extension of $\tau_k$ on~$T_{\alpha,\wock}$ at the next level $f(n+1)$. Thus the outcome of~$\alpha$ in~$\xi_{\wock}$ is at most $\w k +i$; but we know that $\xi\le \xi_{\wock}$, so we must have $\alpha\cat (\w k+i)\prec \xi_{\wock}$. 

It follows that $X_{\wock} = \lim_{s\to \wock} X_s$, where $X_s = \bigcup_{\alpha\prec \xi_s} X_\alpha$ for all $s\le \wock$. Let $A\in \bigcap_{\alpha\prec \xi_\wock} [T_{\alpha,\wock}]$. 

\subsubsection{Hitting the closed set} 
Let $\xi = \xi_{\wock}$ and write $T_\alpha$ for $T_{\alpha,\wock}$. For all $\alpha\prec \xi$, for all $B\in [T_{\alpha}]$, $\+P^B\cap [X_\alpha]\ne \emptyset$. The argument is as in previous constructions. By induction; if $\alpha\cat (\w k+i)\prec \xi$ then $X_{\alpha\cat (\w k+i)} = X_{\alpha}\cat 0$, and as $\tau_k\notin B_{\alpha,\wock}$, there is no $\s$ extending~$\tau_k$, thus no $\s\in T_{\alpha\cat (\w k+i)}$ such that $\+P^\s\cap [X_\alpha\cat 0] = \emptyset$. If $\alpha\cat \w^2\prec \xi$ then $\Nar_{|\alpha|}(T_\alpha)\subseteq B_{\alpha,\wock}$, so every $\tau\in \Nar_{|\alpha|}(T_\alpha) = T_{\alpha\cat \w^2}$  can be extended to some $B$ such that $\+P^B\cap [X_\alpha\cat 0]= \emptyset$; by induction, for all $\tau\in T_{\alpha\cat \w^2}$, $\+P^\tau\cap [X_\alpha\cat 1] \ne \emptyset$. We conclude that $X = X_{\wock}\in \+P^A$. 

\subsubsection{Computing~$\xi$} 
As in previous constructions, we show that $\xi \le_{\wock\Tur} A$ (and it is immediate that $X\le_{\wock\Tur} \xi$). As above, this amounts to showing that if $o<p\le \w^2$, $\alpha\cat o\prec \xi_s$ and $\alpha\cat p\prec \xi$, then $\stem(T_{\alpha\cat o,s})\nprec A$. Let $k,i$ such that $o = \w k +i$. We know that $\stem(T_{\alpha\cat o,s})\succeq \s_i(\tau_k)$, so we show that $\s_i(\tau_k)\nprec A$.

\smallskip

If $p = \w^2$ then we chose $\s_i(\tau_k)\notin T_{\alpha\cat \w^2,s}$; and $A\in [T_{\alpha\cat \w^2}]$ and $T_{\alpha\cat \w^2}\subseteq T_{\alpha\cat \w^2,s}$. 

\smallskip

Suppose that $p = \w m +j$. We show that $\s_j(\tau_m)\perp \s_i(\tau_k)$. We cannot have $\tau_m\prec \tau_k$ because $k\le m$ and so $|\tau_k|\le |\tau_m|$. We cannot have $\tau_m \succeq \s_i(\tau_k)$, because $\tau_m\in T_{\alpha\cat \w^2}$ and $\s_i(\tau_k)\notin T_{\alpha\cat \w^2,s}$. If $\tau_m\perp \s_i(\tau_k)$ we are done. If not, then $\tau_m \prec \s_i(\tau_k)$ and so $\tau_m = \tau_k$, that is, $m=k$. In this case we have $i\ne j$, and $\s_j(\tau_k)$ and $\s_i(\tau_k)$ are distinct strings of the same length, and so are incomparable. (Also, because $i<j$, $\s_i(\tau_k)\notin T_{\alpha}$, whence $\s_i(\tau_k)\nprec A$.)
\end{proof}

A modification of the proof above along the lines of {the} construction giving \cref{thm:existence_of_bad_for_uniform_higher_self_PA} also gives:

\begin{theorem} \label{thm:bad_random_for_self_PA}
	There is a higher $\ML$-random oracle which is bad for uniform self-$\PA$.  
\end{theorem}

\section{Relativising higher randomness} 
\label{ssec:relativising_higher_randomness}

Continuous relativisation of higher Martin-L\"of randomness were considered in \cite{BGMContinuousHigherRandomness} when studying bases for higher ML-randomness and the van Lambalgen theorem. As expected, a \emph{higher $A$-$\ML$-test} is a sequence $\seq{\+U_n}$ of uniformly higher $A$-c.e.\ open sets~$\+U_n$ with $\leb(\+U_n)\le 2^{-n}$; we analgously define higher $A$-ML null sets and higher $A$-ML random sequences. 

In this section we focus on some positive results: ways in which relativised higher randomness mimics standard randomness.

\subsection{ML and Kurtz randomness} 
\label{sub:ml_and_kurtz_randomness}

\begin{definition} \label{def:relative_Kurtz}
	Let $A\in 2^\w$. A sequence is \emph{higher $A$-Kurtz random} if it is an element of no null higher $A$-effectively closed set. 
\end{definition}

\begin{proposition} \label{prop:ML_implies_Kurtz}
	For any~$A$, every higher $A$-$\ML$ random sequence is higher $A$-Kurtz random.
\end{proposition}

\begin{proof}
	We need to cover a higher $A$-effectively closed and null set by a higher $A$-ML null set. 
	We are given a higher oracle effectively closed~$\+P$ such that $\+P^A$ is null. Let $\epsilon>0$; we show how (uniformly in~$\epsilon$) we can find a higher $A$-c.e.\ open set $\+U^A$ such that $\leb(\+U^A)\le \epsilon$ and $\+P^A\subseteq \+U^A$. We mimic the proof that there is no higher self-PA degree, now throwing in measure into the mix.

	\smallskip
	
	First, we tackle the incomplete case, which is an elaboration on \cref{prop:no-self-PA:incomplete_case}. Again let $p\colon\wock\to \w$ be an injective enumeration of~$O$. At stage~$s$, for each~$\tau$, if $\leb(\+P_s^\tau)< \epsilon\cdot 2^{-p(s)}$, then we find a hyperarithmetic open set~$V\supseteq \+P_s^\tau$ with $\leb(V)\le \epsilon 2^{-p(s)}$ and add it to $\+U^\tau_{s+1}$. Thus, for all~$\tau$, 
	\[
		\leb(\+U^\tau) \le \epsilon \sum_{s<\wock} 2^{-p(s)} < \epsilon,
	\]
	so for all~$B$, $\leb(\+U^B)\le \epsilon$. Now suppose that $\+P^A\nsubseteq \+U^A$. This means that for all~$s$,~$m$, and $\tau\prec A$, if $\leb(\+P_s^\tau)< \epsilon 2^{-m}$ then $O_s\rest{m} = O\rest{m}$. It follows that $A\ge_{\wock\Tur}O$. 

	\smallskip
	
	In that case let 
	\[
		 M_\epsilon= \left\{ \s\in 2^{<\w} \,:\,  \leb(\+P^\s) <\epsilon \text{ and $\s$ is minimal such} \right\}. 
	\]
	Note that~$M_\epsilon$ is higher c.e., uniformly in~$\epsilon$, and so is higher $A$-computable, again uniformly. As in the proof of \cref{thm:no_self-PA}, when we see that~$\tau$ believes that some $\s\preceq \tau$ is in~$M_\epsilon$, we wait for the least~$s$ such that $\leb(\+P^\s_s)< \epsilon$, and if found, we find a hyperarithmetic open set $V\supseteq \+P^\s_s$ of measure $\le\epsilon$ and enumerate it into~$\+W^\tau$. Measure will go into $\+W^A$ on behalf of exactly one~$\s\prec A$ (the one in $M_\epsilon$) and one $s<\wock$, so $\leb(\+W^A) \le \epsilon$, and $\+P^A\subseteq \+P^\s \subseteq \+P^\s_s\subseteq \+W^A$. 
\end{proof}

\begin{remark} \label{rmk:covering_Kurtz_null_for_incompletes_by_uniform_test}
	The first part of the proof above can be extended to show that uniformly in~$\epsilon$ we can obtain a higher oracle-c.e.\ open set~$\+U$ such that for all~$B$, $\leb(\+U^B)\le \epsilon$; and for any~$A\nge_{\wock\Tur} O$, $\+U^A$ contains every higher $A$-Kurtz null set. 
\end{remark}

\subsection{Relatively c.e.\ sets are not random} 
\label{sub:relative_c_e_ sets_are_not_random}

It is not difficult to show that higher $A$-computable points cannot be higher $A$-random, indeed not higher $A$-Kurtz random. In fact, we show:

\begin{theorem} \label{thm:ce_sets_are_not_random}
	For all $A$, no higher $A$-c.e.\ set is higher $A$-Kurtz random. 
\end{theorem}

Hirschfeldt first showed that higher $A$-c.e.\ sets cannot be higher $A$-ML random. His argument used Besicovitch density (the law of large numbers). 

\begin{proof}
	Let $X\subseteq \w$ be higher $A$-c.e. If~$X$ is finite then certainly it is not any kind of random. Suppose that~$X$ is infinite; then the collection of supersets of~$X$ is closed and null, and in fact is higher $A$-effectively closed. 
\end{proof}

\subsection{There are relative higher left-c.e.\ randoms} 
\label{sub:there_are_relative_higher_left_c_e_ randoms}

As discussed in the introduction, in the next section we will show that for some very bad oracle~$A$, there is no universal higher $A$-ML test. One of the implications is that it is not clear that there will be a (nonempty) higher $A$-effectively closed set containing only higher $A$-ML random sequences. As a result, the standard proof that for every~$A$, there is an $A$-left c.e., $A$-ML random sequence, does not work in the higher setting: the proof takes the leftmost point in an $A$-effectively closed set of $A$-ML random sequences. 

Nonetheless, we can prove that for all~$A$, there are indeed higher $A$-left c.e., higher $A$-ML random sequences. To prove this, we first need to clarify the notion of higher $A$-left c.e. Recall that the higher left-c.e.\ sequences were defined by the following equivalent conditions:
\begin{enumerate}
	\item $\left\{ \s\in 2^{<\w} \,:\,  \s<X \right\}$ is higher c.e.;
	\item $X$ has a lexicographically non-decreasing $\wock$-computable approximation. 
\end{enumerate}
It is not clear that this equivalence relativises to all oracles, indeed that the second condition is meaningful to relativse for all oracles: for example, if $A$ is tree-collapsing, then it seems that the correct notion of approximation would have length~$\w$, rather than~$\wock$. Thus we define:

\begin{definition} \label{def:relativised_left-c.e.}
	Let $A\in 2^\w$. A sequence $X$ is \emph{higher $A$-left-c.e.}\ if $\left\{ \s\in 2^{<\w} \,:\,  \s<X \right\}$  is higher $A$-c.e.
\end{definition}

Note that for all~$A$, a sequence~$X$ is higher $A$-left-c.e.\ if and only if there is a higher $A$-effectively closed set~$\+P$ whose leftmost element is~$X$. The standard proof applies.

\subsubsection{Working in the unit interval} 
The next proof will work better in the unit interval $[0,1]$ -- its connectedness will be convenient. The standard computable map from Cantor space onto the unit interval: $X\mapsto \sum X(n)2^{-(n+1)}$ -- is not bijective but omitting the binary rationals, gives an isomorphism of effective measure spaces. Call this map~$\Theta$. 

That is, we can define higher oracle effectively open sets $\+U\subseteq 2^\w\times [0,1]$ as those generated by $\Pi^1_1$ sets of pairs $(\s,I)$ where $I\subseteq [0,1]$ is a rational interval, open in $[0,1]$ (so we allow $[0,q)$ and $(q,1]$ for rational $q\in (0,1)$, as well as pairs $(p,q)$ for rational $p<q$ in $[0,1]$), where for such a set~$W$, the open set generated by~$W$ is $\bigcup \left\{[\s]\times I  \,:\, (\s,I)\in W  \right\}$. We can then take sections and for $A\in 2^\w$ talk about higher $A$-c.e.\ open subsets of the unit interval, and using the usual Lebesgue measure on the unit interval (which we also denote by~$\lambda$), define higher $A$-ML tests and thus randomness. The almost isomorphism $\Theta$ is computable, and this shows that for all $X\in 2^\w$ and $A$, $X$ is higher $A$-ML random if and only if $\Theta(X)$ is higher $A$-ML random. 

Similarly, we say that $r\in [0,1]$ is higher $A$-left-c.e.\ if $\left\{ q\in \Rat\cap [0,1] \,:\,  q<r \right\}$ is higher $A$-c.e. Then for all $X\in 2^\w$, $X$ is higher $A$-left-c.e.\ if and only if $\Theta(X)$ is higher $A$-left-c.e.

\subsubsection{Building a higher $A$-left-c.e.\ random} 

\begin{theorem} \label{thm:higher_left_c.e._random}
For every~$A$, there is a higher $A$-left-c.e., higher $A$-$\ML$-random sequence. Such a sequence can be found uniformly in~$A$. 
\end{theorem}

\begin{proof}
As mentioned above, it suffices to show that for all~$A$, there is a higher $A$-left c.e., higher $A$-ML random real in $[0,1]$. As such a real~$r$ is never a binary rational, from~$r$ we can effectively find its binary {expansion} $X\in 2^\w$. 

Fix an effective and acceptable list $\+W_1,\+W_2,\dots$ of all oracle c.e.\ open subsets of $[0,1]$. Given a string~$\s$, we let $r^\s$ be the supremum of all rational $q\in [0,1]$ such that for some $n\ge 1$ and some sequence $C_1,C_2,\dots, C_n$, we have:
\begin{enumerate}[label=(\arabic*)]
	\item each $C_k$ is the union of finitely many rational open subintervals of $[0,1]$;
	\item $C_k\subseteq \+W_k^\s$ for all $k=1,\dots, n$; 
	\item $\leb(C_k)\le 2^{-k}/2$ for all $k=1,\dots, n$; and
	\item $[0,q)\subseteq \bigcup_{k=1}^n C_k$. 
\end{enumerate}
Note that $\leb(C_k)\le 2^{-k}/2$ (and $k\ge 1$) implies that $r^\s\le 1/2$ for all~$\s$. Now 
\[
	\left\{ (q,\s)\,:\, q\in \Rat\cap [0,1] \andd  q<r^\s \right\}
\]
is $\Pi^1_1$. Thus, for all~$A\in 2^\w$, 
\[
	r^A = \sup_{\s\prec A} r^\s
\]
is higher $A$-left-c.e., uniformly in~$A$; and $r^A\le 1/2$ (what is important is that $r^A<1$). We show that for all~$A$, $r^A$ is higher $A$-ML random. 

\medskip

Fix $A\in 2^\w$, and suppose that $r^A$ is not higher $A$-ML random. Let $\seq{\+V_n}$ be a uniform sequence of higher oracle $A$-c.e.\ open sets such that $\seq{\+V_n^A}$ is a higher $A$-ML test which captures $r^A$. 

\begin{smallclaim} \label{smallclaim:small_index}
	There are some $k<m$ such that $\+W_k = \+V_m$. 
\end{smallclaim}
	
\begin{proof}
	Let $\seq{\+B^0_e}_{e<\w}, \seq{\+B^1_e}_{e<\w},\dots$ be a list of all lists of uniform higher oracle-c.e.\ open sets. There is a computable function $f$ such that for all~$n$ and~$e$, $\+W_{f(n,e)} = \+B^n_e$. Let $g$ be a computable function such that $g(e)> f(e,n)$ for all $n\le e$. Fix~$n$ such that $\+B^n_e = \+V_{g(e)}$. Choose $m = g(e)$ for any~$e\ge n$, and $k = f(n,e)< g(e)$. 
\end{proof}

Fix some $k<m$ given by \cref{smallclaim:small_index}. Since $r^A\in \+V_m^A$, find some rational interval $(p,q)\subseteq \+V_m^A$ such that $p<r^A<q$. Fix some $\s\prec A$ such that $r^\s>p$. Find $C_1,C_2,\dots, C_n$ which witness that $r^\s>p$. Now $\+W_k^\s\subseteq \+W_k^A = \+V_m^A$ so $\leb(\+W_k^\s)\le 2^{-m} \le 2^{-k}/2$; so $\leb(C_k\cup (p,q))\le 2^{-k}/2$. Thus replacing $C_k$ by $C_k\cup (p,q)$ in the list $C_1,\dots, C_n$ witnesses that $r^\s\ge q$, which is impossible. 
\end{proof}

\section{Universal higher $A$-ML tests} \label{sec:bad_oracle_universal_mltest}

A higher $A$-ML test $\seq{\+V_n}$ is \emph{universal} if $\bigcap_n \+V_n$ contains every higher $A$-ML non-random sequence. An oracle~$A$ is \emph{bad for universal tests} if there is no universal higher $A$-ML test. Below we construct an oracle bad for universal tests. Before we do so, we exhibit some classes of good oracles.

\subsection{Uniform higher ML-tests} 
\label{sub:uniform_higher_ml_tests}

We remark that we have already constructed a weak version of a bad oracle. If $\seq{\+V_n}$ is a higher $A$-ML test then there is a uniform list $\seq{\+U_n}$ of higher oracle c.e.\ open sets such that $\+V_n = \+U_n^A$. For other oracles~$B$, the sequence $\seq{\+U_n^B}$ may fail to be a higher $B$-ML test, because for some~$n$ we may have $\leb(\+U_n^B)>2^{-n}$. We say that $\seq{\+U_n}$ is a (uniform) \emph{higher oracle $\ML$-test} if for all~$B$, $\seq{\+U_n^B}$ is a higher $B$-ML test. 

\Cref{th:no_A_universal_oracle_continuous_test} (together with the fact that higher $A$-computable sets are not higher $A$-ML random) implies that for some~$A$, for no higher uniform test $\seq{\+U_n}$ is $\seq{\+U_n^A}$ universal for higher $A$-ML randomness. That is, unlike the lower setting, there is no oracle ML-test which is universal for every oracle. This is strengthened by \cref{th:no_A_unniversal_mlr_test}.





\subsection{Good oracles for universal tests} 
\label{sub:good_oracles_for_universal_tests}

\begin{theorem} \label{thm:good_oracles_for_universal_tests}
	Suppose that one of the following holds for~$A$:
	\begin{enumerate}[label=\textup{(\alph*)}]
		\item $A\ge_{\wock\Tur} O$;
		\item $A$ is tree-collapsing; or
		\item $A$ is higher $\ML$-random. 
	\end{enumerate}
	Then there is a universal higher $A$-$\ML$ test. 
\end{theorem}

\begin{proof}
We consider each class in turn. In all three cases, the key for constructing a universal test for~$A$ is the ability to uniformly in~$e$ and~$\epsilon>0$ produce a higher $A$-c.e.\ open set $\+U^A_{e,\epsilon}$ such that $\leb(\+U_{e,\epsilon}^A)\le \epsilon$ and if $\leb(\+W_e^A)\le \epsilon$ then $\+U^A_{e,\epsilon} = \+W_e^A$. Here of course $\seq{\+W_e}$ is an efective listing of all higher oracle c.e.\ open sets.

\subsubsection{When $A\ge_{\wock\Tur} O$} 

We elaborate on the previous proofs using higher complete oracles (\cref{thm:no_self-PA,prop:ML_implies_Kurtz}). Fix~$e$ and~$\epsilon$. We let 
\[
	M_{e,\epsilon} = \left\{ \s\in 2^{<\w} \,:\, \leb(\+W_e^\s)\le \epsilon  \right\}. 
\]
Then informally we let, for all $\tau\in 2^{<\w}$, 
\[
	\+U_{e,\epsilon}^\tau = \bigcup \left\{ \+W_e^\s \,:\,  \s\prec \tau \andd \tau\text{ says that }\s\in M_{e,\epsilon}  \right\}. 
\]

\subsubsection{When $A$ is tree-collapsing} 

Let~$T$ witness that~$A$ is tree-collapsing. At stage~$s$, if $\tau\in T_{s+1}\setminus T_s$, and further, $\leb(\+W_{e,s}^\tau)\le \epsilon$, we enumerate all of $\+W_{e,s}^\tau$ into $\+U_{e,\epsilon}^\tau$. 

Note that in this case, the universal higher $A$-ML test is a uniform higher ML-test.

\subsubsection{When $A$ is $\ML$-random} 
This relies on the fact that unrelativsed higher ML-randomness does have a universal test, indeed one constructed using the usual approach. This means it is \emph{effectively universal}; what it means for us is that for any higher ML-random~$Z$ there is a computable function~$f_Z$ such that given an index~$e$ for a higher ML-test $\seq{\+V_n}$, we have $Z\notin \+V_{f_Z(e)}$. 

So now for each~$e$, $\epsilon$ and~$n$, we enumerate a higher oracle-c.e.\ open set $\+U_{e,\epsilon,n}$ such that for all $B\in 2^\w$, if $\leb(\+W_e^B)\le \epsilon$, then $\+U_{e,\epsilon,n}^B = \+W_e^B$, and such that \[
	\leb \left( \left\{ B\in 2^\w \,:\, \leb(\+U_{e,\epsilon,n}^B) > \epsilon  \right\}  \right) \le 2^{-n};
\]		
note that the set of such~$B$'s is higher c.e.\ open, uniformly. Thus a random~$Z$ will satisfy $\leb(\+U_{e,\epsilon,n}^Z)\le \epsilon$ for all but finitely many~$n$, and further, we can find such~$n$ effectively from~$e$ and~$\epsilon$.

The construction of $\+U_{e,\epsilon,n}$ follows similar constructions in \cite{BGMContinuousHigherRandomness}. At stage~$s$ let 
\[
	Q_s = \left\{ B\in 2^\w \,:\,  \leb(\+W_{e,s}^B) \le \epsilon \right\}.
\]
This is a $\Delta^1_1$ closed set, uniformly in~$s$. We find a $\Delta^1_1$ open set $V_s\supseteq Q_s$ such that $\leb(V_s\setminus Q_s) < 2^{-n}2^{-p(s)}$ where again $p\colon \wock\to \w$ is injective. We let 
\[
	\+U_{e,\epsilon,n}^B = \bigcup \left\{ \+W_{e,s}^B \,:\,  s<\wock \andd B\in V_s \right\}. \qedhere
\]
\end{proof}

\subsection{A bad oracle for universal tests} 
\label{sub:bad_oracle_for_universal_tests}

Finally: 

\begin{theorem} \label{th:no_A_unniversal_mlr_test}
There exists an oracle~$A$ for which there is no universal higher $\ML$-test. Such~$A$ can be made higher $\Delta^0_2$ and such that $\w_1^A>\wock$. 
\end{theorem}	

Indeed, we construct an oracle which is ``self-PA for randomness'':

\begin{theorem} \label{th:no_A_continuous_mlr_test}
	There is an oracle~$A$ such that every nonempty higher $A$-effectively closed set contains some sequence which is not higher $A$-$\ML$ random. 
\end{theorem}

Again, the~$A$ constructed is higher $\Delta^0_2$ and collapses~$\wock$. Further, every higher $A$-effectively closed set contains some higher $\Delta^0_2$ sequence which is not higher $A$-ML random. 
The rest of this section is dedicated to the proof of \cref{th:no_A_continuous_mlr_test}.

\subsection{Informal discussion and basic ingredients} 
\label{sub:general_idea_of_the_construction}

How would we modify the construction of \cref{thm:existence_of_bad_for_uniform_higher_self_PA} to prove \cref{th:no_A_continuous_mlr_test}? Let $\+P$ be a higher oracle effectively closed set. Consider the first step toward dealing with~$\+P$, and suppose that we are working above some tree $T\in \+T$. As above, we would set $X(0)=0$ and try the various subtrees $T\dbrak{k}$ until we get, for each such~$k$, some $\s\in T\dbrak{k}$ such that $\+P^\s\cap [0] = \emptyset$. We then switch to setting $X(0)=1$ and take the narrow subtree $\Nar(T)$. However, it is now possible that there is some $\tau\in \Nar(T)$ such that $\+P^\tau\cap [1]=\emptyset$. In the previous construction, this proved that for some~$B$, $\+P^B=\emptyset$, in which case we didn't need to worry about~$\+P$ at all. Now we cannot dismiss~$\+P$ out of hand; but an obvious solution would be to ensure that $\+P^A$ is empty. This is done by first choosing $\tau\in \Nar(T)$ such that $\+P^\tau\cap [1]=\emptyset$, and some $k$ large enough so that $\s_k(T)\succ \tau$; then choosing some $\s\in T\dbrak{k}$ for which $\+P^\s\cap [0]=\emptyset$ (we may assume $\s\succ \tau$). We make sure that~$A$ extends this~$\s$ by finding some~$m$, sufficiently large, so that $\s_m(T\dbrak{k})\succeq \s$ and routing the construction to work in the tree $T\dbrak{k}\dbrak{m}$. 

This extra step -- going back to extend some $\s_k(T)$ -- creates serious difficulties for the part of the agrument that above gives $\xi\le_{\wock\Tur} A$; and indeed by \cref{thm:no_self-PA} we know that an attempt to get $X\le_{\wock\Tur} A$ must fail. Instead, we make $X$ not random relative to~$A$, which amounts to enumerating, with oracle~$A$, for each~$n$, relatively few possibilities for $X\rest{n}$. Of course for~$\+P$ itself there is nothing left to do, as we ensure that $\+P^A = \emptyset$; however, the effect of this action on other requirements which handle other effectively closed sets can be, if we are not careful, disastrous. 

We cannot get $X\le_{\wock\Tur}A$ because there will be instances in which $\stem(T_{n,s})\prec A$ but $X_{s}\rest{n}\nprec X$. The main driver of this construction is minimising the number of times this happens. 

\subsubsection{Mapping strategies to trees} 

Let $\seq{\+P_e}$ be an effective list of all higher oracle effectively closed sets. We will approximate an oracle~$A$ and for each~$e$, a sequence~$X_e$ and ensure that the $e\tth$ requirement is met:
\begin{itemize}
	\item If $\+P_e^A$ is nonempty, then $X_e\in \+P_e^A$, and $X_e$ is not higher $A$-ML random.
\end{itemize}

The collection of strategies is now $(\w+2)^{<\w}$ (with the new outcome $\w+1$ indicating that we are going back to make $\+P_e^A$ empty). For each~$e$ and each strategy we will define $X_{e,\alpha}$, again managing lengths so that $X_{e,\alpha}(d)$ is defined if and only if $\seq{e,d}<|\alpha|$; for $n= \seq{e,d}$, $X_{e,\alpha}(d)=0$ if $\alpha(n)<\w$, and $X_{e,\alpha}(d)=1$ otherwise. 

\emph{However}, we will shortly show that we need to be judicious about the choice of pairing function $(e,d)\mapsto \seq{e,d}$. 

\smallskip

A new ascpect of our mapping of strategies to trees will be that $T_\alpha$ will not be defined for all strategies~$\alpha$; rather, it will be defined for all strategies that are accessible during some stage of the construction. The issue is with assigning a tree to the outcome $\w+1$, which we can only do only once we found strings on previous trees that allow us to make some $\+P_e^A$ empty. In the meantime, we define:
\begin{itemize}
	\item $T_{\epsilon} = 2^{<\w}$;
	\item If $T_\alpha$ is defined, then $T_{\alpha\cat k} = \Nar(T_\alpha\dbrak{k})$;
	\item If $T_\alpha$ is defined, then $T_{\alpha\cat \w} = \Nar(T_\alpha)$;
	\item If $\alpha\cat (\w+1)$ is ever accessible then we will define $T_{\alpha\cat(\w+1)} = T_\alpha\dbrak{k}\dbrak{m}$ for some $k,m<\w$. 
\end{itemize}
Note that we let $T_{\alpha\cat k} = \Nar(T_\alpha \dbrak{k})$ rather than $T_\alpha\dbrak{k}$. This ensures that in all cases, if $T_\alpha$ and~$T_\beta$ are defined and $\alpha\perp \beta$ then $[T_\alpha]\cap [T_\beta]=\emptyset$.

\subsubsection{Counting and bounding} 

Suppose that $\stem (T_\alpha)\prec A$. Then there are at most three outcomes $o<\w+2$ such that $\stem(T_{\alpha\cat o})\prec A$, namely because $k\ne l <\w$ implies $\stem(T_{\alpha\cat k})\perp \stem(T_{\alpha\cat l})$. However, for some fixed~$e$, $X_{e,\alpha\cat o}$ extends $X_{e,\alpha}$ by at most one bit, so this bound on the number of outcomes will give us no bound at all on the number of possible versions of $X_e\rest{n}$ that get enumerated by~$A$. 

The main aim of the construction, \cref{lemma:final_cor_univ_test}, is to ensure that if $\+P_e^A$ is nonempty, $\alpha$ is an \emph{$e$-strategy}, which means that $|\alpha| = \seq{e,d}$ for some~$e$, $\alpha$ is accessible at some stage, and $\stem(T_\alpha)\prec A$, then there is at most \emph{one} possible outcome~$o$ such that $\alpha\cat o$ is accessible at some stage and $\stem(T_{\alpha\cat o})\prec A$. Note that if~$\alpha$ is an $e$-strategy then $|X_{e,\alpha\cat o}| = |X_{e,\alpha}|+1$. 

This informs us how to choose the pairing function $(e,d)\mapsto\seq{e,d}$. What we need is a long stretch of $e$-strategies, which would not be disturbed by other strategies in their midst. For suppose that~$n>m$, and that every $k$ in the interval $[m,n)$ is of the form $\seq{e,d}$ for some~$d$. By induction, there are at most $3^m$ many strategies $\beta$ of length~$m$ such that $\stem(T_\beta)\prec A$. However, if we achieve \cref{lemma:final_cor_univ_test}, then there will be at most $3^m$ many strategies~$\alpha$ of \emph{length~$n$} such that $\stem(T_\alpha)\prec A$; and for each such strategy we have $|X_{e,\alpha}|\ge n-m$. So if we take~$n$ much larger than~$m$ this will allow us to capture~$X_e$ in a set of small measure. 

We thus partition the natural numbers into the intervals $[4^n,4^{n+1})$, and assign each interval a number~$e$, in such a way that every number is assigned infinitely many intervals. We then let $\seq{e,d}$ be the $d\tth$ element of the union of all intervals assigned to~$e$. Note that we preserve the requirement
\begin{itemize}
	\item $\seq{e,d}<\seq{e,d+1}$ for all~$e$ and~$d$. 
\end{itemize}

\subsubsection{Noise cancelling} 

The question remains, how do we achieve the ``noise-cancelling'' \cref{lemma:final_cor_univ_test}? Here our construction starts to deviate significantly from those we had above. Suppose that~$\alpha$ is an $e$-strategy, and that $\+P_e^A$ is nonempty. We also assume that we are working after the stage at which the shortest~$e$-strategy has stabilised. Our first mission is to ensure:
\begin{itemize}
 	\item $\alpha\cat (\w+1)\nprec \xi$. 
 \end{itemize} 
The basic idea is that if $\alpha\cat (\w+1)\prec \xi$ then we can make $\+P_e^A\cap [X_{\alpha}]=\emptyset$; but then, by induction, we either would have moved away from~$\alpha$, or had the means to erase the last bit of~$X_\alpha$ and get an even larger chunk taken out of $\+P_e^A$; eventually, by induction, we would get $\+P_e^A=\emptyset$. 

The next step is to show:
\begin{itemize}
	\item If $k<\w$, $\alpha\cat k$ is accessible at some stage, but also $\alpha\cat o$ is accessible at some stage with $o>k$, then $\stem(T_{\alpha]\cat k})\nprec A$. 
\end{itemize}

This will give us what we want: at most one accessible outcome~$o$ of~$\alpha$ would give a stem comparable with~$A$. How do we achieve this? The difficulty is with some $\beta\prec \alpha$ taking the outcome $\w+1$, and instructing us to take a string extending $\s_k(T_\alpha)$ in order to make $\+P_{e'}^A$ empty for some $e'\ne e$. 

This would be quite bad, and so we see that~$\beta$ cannot immediately take the outcome $\w+1$ when it sees any string $\tau\in \Nar(T_\beta)$ with $\+P_{e'}^\tau\cap [1]=\emptyset$. On the other hand, it does not actually need to: even if $\beta\prec \xi$, we only care about whether $\+P_{e'}^A$ is empty or not, not whether $\+P_{e'}^B$ is empty for some other $B\in [T_\beta]$. At every stage~$s$, $\seq{T_\alpha}_{\alpha\prec \xi}$ (induces) a shrinking path in $\+T$, and so determines an approximation $A_s$ to~$A$; we only need to act when we see some $\tau\prec A_s$ making $\+P_{e'}^\tau\cap [1]$ empty. If $A_s$ has already moved away from $\s_k(T_\alpha)$, then $\tau$ will not extend that string. That is the main idea of how to achieve noise cancelling. It does mean that our construction looks quite different from previous ones, as at each stage we only investigate what happens along $A_s$, not on an entire tree. 

In turn, this creates some difficulties. Consider, for example, the second step of determining $T_{\beta\cat \w}$, namely, after we found~$k$, we need to find~$\s\in \Nar(T_{\beta}\dbrak{k})$ such that $\+P_{e'}^\s\cap [0]=\emptyset$ and some~$m$ such that $\s_m(T_\beta\dbrak{k})$ extends~$\s$. And again, we do not want to extend some string $\s_l(T_{\alpha})$; we need to find such~$\s$ which is an initial segment of some historical version~$A_t$ of~$A$. So that we have such a version to point to, we need, for example, at times, to let~$A$ pass through strings we know at the time are not going to be actual initial segments of~$A$. This is quite counter-intuitive compared to most constructions in computability, but appears to be neccessary in this one.

\subsection{The construction}

\subsubsection{Construction claims}

Before starting the construction we make a few claims which we will make use of during the construction; we will verify them after we state the construction. 

\begin{enumerate}[label=\textbf{Claim \arabic*.},align=left,ref={Claim \arabic*}]

\item \label{Univ:ConstClaim:left-c.e.}
For all $t<s<\wock$, $\xi_t\le \xi_s$. 

\item \label{Univ:ConstClaim:PreviousOutcomesAccessible}
For every $\alpha$, $s$ and outcomes~$p<o\le \w+1$, if $\alpha\cat o\prec \xi_s$ then there is some $t<s$ such that $\alpha\cat p\prec \xi_t$. 

\item \label{Univ:ConstClaim:Shrinking}
For each stage~$s$, the sequence $\{T_{\alpha}\}_{\alpha \prec \xi_s}$ induces a shrinking path of~$\T$, and so $\bigcap_{\alpha\prec \xi_s} [T_\alpha]$ is a singleton $\{A_s\}$. 

\item \label{Univ:ConstClaim:AccessibleClosed}
For each~$\alpha$, $\left\{ s<\wock \,:\,  \alpha\prec \xi_s \right\}$ is a closed interval of~$\wock$. 

\item \label{Univ:ConstClaim:last_k_stage}
For each $e$-strategy~$\alpha$ and $k<\w$, if~$\alpha\cat k\prec \xi_s$, $\alpha\prec \xi_{s+1}$ but $\alpha\cat k\nprec \xi_{s+1}$, then $\+P_{e,s}^{A_s}\cap [X_{e,\alpha}\cat 0]=\emptyset$. 

\item \label{Univ:ConstClaim:omegaplusone_outcome}
For each $e$-strategy~$\alpha$ and~$s$, if $\alpha\cat (\w+1)\prec \xi_s$ then $\+P_{e,s}^{\rho}\cap [X_{e,\alpha}]=\emptyset$ for $\rho = \stem(T_{\alpha\cat{(\w+1)}})$.

\end{enumerate}

By induction on $s<\wock$ we define $\xi_s\in (\w+1)^\w$, and for all $\alpha$ which is ever accessible, a tree $T_\alpha\in \+T$, following the scheme described above.

\subsubsection{The construction:}

At stage $s=0$ we set $\xi_0 = 0^\infty$. 

\medskip

Suppose that $s<\wock$ is a limit stage. Let~$n$ be the largest ordinal $n\le \w$ such that $\seq{\xi_t\rest{n}}_{t<s}$ is stable on some final segment of~$s$. Determine that this stable version $\lim_{t\to s} \xi_t\rest{n}$ is an initial segment of $\xi_s$. If $n<\w$, let $\xi_s(n) = \w = \sup_{t<s} \xi_t(n)$, and for $m>n$ let $\xi_s(m)=0$. Note that $T_\alpha$ is defined for all $\alpha\prec \xi_s$, as no new value $\w+1$ is introduced at this stage.

\medskip

Now, let $s<\wock$, and suppose that $\xi_s$ has already been defined, along with $T_\alpha$ for all $\alpha\prec \xi_s$. We define $\xi_{s+1}$. 

For each $e<\w$ let $\alpha_{e,s} = \xi_s\rest{\seq{e,0}}$ be the shortest $e$-strategy $\prec \xi_s$. An $e$-strategy $\alpha\prec \xi_s$ \emph{has a problem} at stage~$s$ if $\alpha_{e,s}\cat (\w+1)\nprec \xi_s$, and one of the following holds:
\begin{enumerate}[label=(\arabic*)]
  	\item $\alpha\cat k\prec \xi_s$ for some $k<\w$, and $\+P_{e,s}^{A_s}\cap[X_{e,\alpha}\cat 0]=\emptyset$.
  	\item $\alpha\cat \w\prec \xi_s$, and $\+P_{e,s}^{A_s}\cap[X_{e,\alpha}\cat 1]=\emptyset$.
  	\item $\alpha\cat (\w+1)\prec \xi_s$.
\end{enumerate}  

If no initial segment of $\xi_s$ has a problem at stage~$s$ then we let $\xi_{s+1} = \xi_s$. Suppose otherwise. Our instructions are not quite to deal with the shortest initial segment of~$\xi_s$ which has a problem at stage~$s$. Rather, we prioritse case~(3). That is, we choose~$\alpha\prec \xi_s$ which has a problem at stage~$s$, and among all $\tilde\alpha\prec \xi_s$ which have a problem at stage~$s$, 
\begin{itemize}
	\item $\alpha$ is the shortest with $\alpha\cat(\w+1)\prec \xi_s$; or,
	\item if there is no $\tilde \alpha$ which has a problem at stage~$s$ and $\tilde \alpha\cat (\w+1)\prec \xi_s$, then we let~$\alpha$ be the shortest which has a problem at stage~$s$. 
\end{itemize}

We act according to the cases above. 

\smallskip

In case~(1), we let $\alpha\cat(k+1)\prec \xi_{s+1}$, and for all $m>|\alpha|$ we let $\xi_s(m)=0$. 

\smallskip

In case~(2), we let $\alpha\cat (\w+1)\prec \xi_{s+1}$, and for all $m>|\alpha|$, we let $\xi_s(m)=0$. 

\smallskip

In case~(3), we know that $\alpha$ is a proper extension of $\alpha_e$. Let $\beta\prec \alpha$ be the longest $e$-strategy before~$\alpha$. That is, $\beta = \alpha\rest{\seq{e,d-1}}$, where $|\alpha| = \seq{e,d}$. Let $o$ be such that $\beta\cat o\prec \xi_s$; by minimality of~$\alpha$, we know that $o\ne \w+1$. We let $\beta\cat (o+1)\prec \xi_{s+1}$, and for $m>|\beta|$, we let $\xi_{s+1}(m)=0$. 

\smallskip

In case~(2) and in case~(3) when $\beta\cat \w\prec \xi_s$, the new outcome chosen is~$\w+1$. Let $\gamma = \alpha$ in case (2), and $\gamma = \beta$ in case~(3) when $\beta\cat \w\prec \xi_s$. By \ref{Univ:ConstClaim:left-c.e.}, the strategy $\gamma\cat(\w+1)$ was not accessible at any stage $t\le s$. Hence we now need to define $T_{\gamma\cat(\w+1)}$. In either case~(2) or~(3), $\gamma\cat\w\prec \xi_s$. Also, $\+P_{e,s}^{A_s}\cap [X_{e,\gamma}\cat 1] = \emptyset$; in case~(2) by choice of $\gamma=\alpha$, in case~(3) by applying \ref{Univ:ConstClaim:omegaplusone_outcome} to~$\alpha$, and noting that $\beta\cat \w\preceq \alpha$ means that $X_{e,\alpha} = X_{e,\beta}\cat 1$. 

We need to find $k,m<\w$ to let $T_{\gamma\cat (\w+1)} = T_\gamma\dbrak{k}\dbrak{m}$. This is done in two similar steps. At the first step, we find some $\tau\prec A_s$ such that $\+P_{e,s}^{\tau}\cap [X_{e,\gamma}\cat 1] = \emptyset$. Recall that each $\s_k(T_\gamma)$ extends an odd level branching node of $T_\gamma$, and that these are dense in~$T_\gamma$; so we find some $\tau'\succeq \tau$, $\tau'\prec A_s$, which is an odd-level branching node of~$T_\gamma$. Note that $A_s\in [\Nar(T_\gamma)] = [T_{\gamma\cat \w}]$; so $\tau'\cat 0\prec A_s$, while for some~$k$, $\tau'\cat 1\preceq \s_k(T_\gamma)$, and $\tau'\cat 1\notin \Nar(T_\gamma)$. The property we are after is:
\begin{itemize}
	\item For all $\rho\in T_{\gamma\cat \w}$, $\rho \prec \s_k(T_\gamma)\Then \rho\prec A_s$. 
\end{itemize}
So this determines~$k$. Now, as $k<\w$ and $\gamma\cat\w\prec \xi_s$, by \ref{Univ:ConstClaim:PreviousOutcomesAccessible}, there is some stage~$t<s$ such that $\gamma\cat k\prec \xi_t$; by \ref{Univ:ConstClaim:AccessibleClosed}, we can choose the last such~$t$. Note that $t<s$ and that $\gamma \prec \xi_{t+1}$ (follows from $\gamma\prec \xi_s$ and \ref{Univ:ConstClaim:left-c.e.}); by \ref{Univ:ConstClaim:last_k_stage}, $\+P_{e,t}^{A_t}\cap [X_{e,\gamma}\cat 0]= \emptyset$. Note that $\s_k(T_\gamma) = \stem (T_{\gamma\cat k}) \prec A_t$; we find some $\s\prec A_t$, $\s\succeq\s_k(T_\gamma)$, such that $\+P_{e,t}^\s\cap [X_{e,\gamma}\cat 0] = \emptyset$. We then repeat the same method to find~$m$: we note that $A_t\in [T_{\gamma\cat k}] = \Nar(T_\gamma\dbrak{k})$. We find~$m$ such that $\s_m(T_\gamma\dbrak{k})$ extends an odd-level branching point of $T_\gamma\dbrak{k}$ which is an initial segment of $A_t$ which extends~$\s$. Thus, 
\begin{itemize}
	\item For all $\rho\in T_{\gamma\cat k}$, $\rho \prec \s_m(T_\gamma\dbrak{k})\Then \rho\prec A_t$. 
\end{itemize}
We let $T_{\gamma\cat(\w+1)} = T_\gamma\dbrak{k}\dbrak{m}$. 

\smallskip

This concludes the construction.

\subsection{Verification}

\subsubsection{Construction claims} 
Of the six construction claims above, most are immediate by induction on stages and examining the construction. For \ref{Univ:ConstClaim:Shrinking}, as in a previous construction, for each~$e$ such that $\+P_e^B = 2^\w$ for all~$B$, by induction on~$s$, we see that for every $e$-strategy~$\alpha$, $\alpha\cat 0\prec \xi_s$. For \ref{Univ:ConstClaim:omegaplusone_outcome}, we chose $\stem(T_{\gamma\cat(\w+1)}) = \s_m(T_\gamma\dbrak{k})$ which extended both~$\s\prec A_t$ such that $\+P_{e,t}^{\s}\cap [X_{e,\gamma}\cat 0]=\emptyset$ and some $\tau\prec A_s$ such that $\+P_{e,s}^{\tau}\cap [X_{e,\gamma}\cat 1]=\emptyset$. 

The only claim which may need a bit more arguing is \ref{Univ:ConstClaim:last_k_stage}. Let~$\alpha$ and~$k$ be as in the claim. If case~(1) holds at stage~$s$, then the claim is clear from our instructions. However, it is possible that case~(3) holds for the next $e$-strategy $\delta\succ \alpha$: so $\delta\cat(\w+1)\prec \xi_s$, and $\alpha\cat k\preceq \delta$. In this case $X_{e,\delta} = X_{e,\alpha}\cat 0$. By \ref{Univ:ConstClaim:omegaplusone_outcome}, $\+P^\rho_{e,s}\cap [X_{e,\delta}] = \emptyset$ for $\rho = \stem(T_{\delta\cat (\w+1)})$, and as $\delta\cat(\w+1)\prec \xi_s$, we have $\rho\prec A_s$.

\subsubsection{Convergence}
\ref{Univ:ConstClaim:left-c.e.} implies that $\seq{\xi_s}_{s<\wock}$ converges to some $\xi\in (\w+2)^\w$, and so that $\seq{A_s}$ converges to~$A\in \bigcap_{\alpha\prec \xi} [T_\alpha]$ and that each~$\seq{X_{e,s}}$ converges to some~$X_e$, computed from~$\xi$.

\subsubsection{On the outcome $\w+1$} 

As we mentioned above, during the construction we sometimes chose $\alpha\cat(\w+1)\prec \xi_s$ for some~$e$-strategy $\alpha$ which is not the shortest $e$-strategy. This was done so that we could point at~$A_s$ at some later stage; but in the limit, this will not happen, unless $\+P_e^A$ is empty.

For $e<\w$, let $\alpha_e = \alpha_{e,\wock} = \xi\rest{\seq{e,0}}$ be the shortest $e$-strategy $\alpha\prec \xi$. 

\begin{lemma} \label{lem:the_omega_plus_one_outcome}
	Suppose that there is an $e$-strategy $\alpha$ such that $\alpha\cat (\w+1)\prec \xi$. Then $\alpha_e\cat (\w+1)\prec \xi$, and $\+P_e^A = \emptyset$. 
\end{lemma}

\begin{proof}
	We first show that $\alpha_e\cat(\w+1)\prec \xi$. For let~$\beta$ be the shortest $e$-strategy such that $\beta\cat (\w+1)\prec \xi$. Suppose that $\beta\ne \alpha_e$. By \ref{Univ:ConstClaim:left-c.e.}, for all~$s$, $\alpha_e\cat (\w+1)\nprec \xi_s$. By \ref{Univ:ConstClaim:omegaplusone_outcome}, as $\stem(T_{\beta\cat(\w+1)})\prec A_s$ for sufficiently late~$s$, for such~$s$ we have $\+P_{e,s}^{A_s}\cap [X_{e,\beta}]= \emptyset$; eventually, this~$\beta$ will receive attention and cause $\xi_{s+1}$ to move to the right of~$\beta$, which contradicts \ref{Univ:ConstClaim:left-c.e.}. 

	That $\+P_e^A = \emptyset$ follows from $X_{e,\alpha_e} = \epsilon$ and \ref{Univ:ConstClaim:omegaplusone_outcome}. 
\end{proof}

For the next lemma, note that if $\alpha\cat (\w+1)\prec \xi_s$ and $\alpha$ has a problem at stage~$s$ (that is, if $\alpha_{e,s}\cat (\w+1)\nprec \xi_s$), then $\alpha$ lies to the left of $\xi_{s+1}$. This is because such~$\alpha$ have priority at that stage. 

\begin{lemma} \label{lem:the_missing_lemma}
	At each stage~$s$ there is at most one $\alpha\prec \xi_s$ which has a problem at stage~$s$ and $\alpha\cat(\w+1)\prec \xi_s$. 
\end{lemma}

\begin{proof}
	At a limit stage~$s$ there is no such~$\alpha$. Suppose that~$\alpha\prec \xi_s$ shows otherwise. There is some $t<s$ such that $\alpha\cat(\w+1)\prec \xi_t$. Then $\alpha_{e,s} = \alpha_{e,t}$, and $\alpha_{e,t}\cat(\w+1)\nprec \xi_t$ (as it is $\nprec \xi_s$), so~$\alpha$ has a problem at stage~$t$ and we would route $\xi_{t+1}$ to the right of~$\alpha$. 

\smallskip

	Let~$s$ be a stage and suppose that $\alpha\cat(\w+1)\prec \xi_{s+1}$, and has a problem at stage~$s+1$. Then $\alpha\cat\w\prec \xi_s$ (or $\alpha$ would have a problem at stage~$s$ and we would route $\xi_{s+1}$ to the right of~$\alpha$). No $\delta\prec \alpha$ has $\delta\cat(\w+1)\prec \xi_s$ and has a problem at stage~$s$, or again we would have $\alpha<\xi_{s+1}$. Hence there is no such~$\delta\prec \alpha$ at stage~$s+1$ as well. Also, for all $m\ge |\alpha|$, $\xi_{s+1}(m)\ne \w+1$, which shows the uniqueness of~$\alpha$ at stage~$s$. 
\end{proof}

\subsubsection{The sequences $X_e$}

\begin{lemma} \label{lem:MLR:X_e_in_P_A}
	For all~$e$, if $\+P_e^A\ne\emptyset$ then $X_e\in \+P_e^A$. 
\end{lemma}

\begin{proof}
	By \cref{lem:the_omega_plus_one_outcome}, for every $e$-strategy $\alpha\prec \xi$, $\alpha$'s outcome $o$ (that is, $\alpha\cat o\prec \xi$) satisfies $o\le \w$. 

	If $o = k<\w$ then by our construction, $\+P_e^A\cap [X_{e,\alpha}\cat 0]\ne \emptyset$: otherwise, for some $\s\prec A$ we have $\+P_e^\s\cap [X_{e,\alpha}\cat 0] = \emptyset$ and eventually $\s\prec A_s$, which would prompt us to move away from~$\alpha\cat k$. 

	If $o=\w$ then similarly, $\+P_e^A\cap [X_{e,\alpha}\cat 1]\ne\emptyset$. In either case, $\+P_e^A\cap [X_{e,\alpha}]\ne \emptyset$. Further, $X_e = \bigcup \left\{ X_{e,\alpha} \,:\,  \alpha\prec \xi \text{ is an $e$-strategy} \right\}$, so the result follows by closure.
\end{proof}

\subsubsection{The noise canceling lemma}

Toward capturing $X_e$ in a higher $A$-ML test, we prove a ``noise cancelling'' lemma. For brevity, let~$\Access$ be the collection of strategies which are ever accessible:
\[
	\Access = \left\{ \alpha\in (\w+2)^{<\w} \,:\,  (\exists s<\wock) \,\,\alpha\prec \xi_s \right\}.
\]

\begin{lemma} \label{lemma:universal_test_main_lemma}
Let~$\alpha$ be a strategy, and let $o\le \w$. Suppose that:
\begin{enumerate}
	\item $\alpha\cat o\in \Access$; and 
	\item $\alpha\cat(\w+1)\notin \Access$. 
\end{enumerate}
Then for all $k<o$, $A\nsucc \stem(T_{\alpha \cat k})$. 
\end{lemma}

\begin{proof}
First let us emphasize that the only ``bad'' case is that when for some $m<n$, the bit $\xi(m)$ takes the value $\omega+1$ at some stage; this will be observed in the proof. 

\smallskip

Let~$t$ be a stage such that $\alpha\cat o\prec \xi_t$, and let $k<o$. We prove by induction on $s\ge t$ that $A_s\nsucc \stem(T_{\alpha\cat k}) = \s_k(T_\alpha)$. 

\medskip

This holds for~$s=t$: either $o<\w$, in which case $\s_o(T_\alpha)\prec A_t$, and $\s_o(T_\alpha)\perp \s_k(T_\alpha)$; or $o=\w$, in which case $A_t\in \Nar(T_\alpha)$ and $\s_k(T_\alpha)\notin \Nar(T_\alpha)$ (see \cref{fact:small_tool1}). In fact, since $\alpha\cat (\w+1)$ is never accessible, this argument shows that for all $s\ge t$, if $\alpha\prec \xi_s$ then $\s_k(T_\alpha)\nprec A_s$. 

\medskip

Let $s>t$, and suppose that for all $r\in [t,s)$, $\s_k(T_\alpha)\nprec A_r$. We show that $\s_k(T_\alpha)\nprec A_s$. 

\smallskip

Let $\beta = \alpha\wedge \xi_s$ be the longest common initial segment of~$\alpha$ and~$\xi_s$. By the argument above, we may assume that $\beta\ne \alpha$. Let~$p<q\le \w+1$ such that $\beta\cat p\preceq \alpha$ and $\beta\cat q\prec \xi_s$. 

\smallskip

Suppose that $q\le \w$; so $p<\w$. In this case we have $\s_k(T_\alpha)\succ \s_p(T_\beta)$. As above, if $q<\w$  then $\s_q(T_\beta)\prec A_s$ (and $\s_p(T_\beta)$ and $\s_q(T_\beta)$ are incomparable); if $q=\w$ then $A_s\in \Nar(T_\beta)$ so again $\s_p(T_\beta)\nprec A_s$. 

So as discussed, the interesting case is when $q = \w+1$. So $p\le \w$. Let $c,d<\w$ be such that $T_{\beta\cat(\w+1)} = T_\beta\dbrak{c}\dbrak{d}$, and let~$r$ be the least stage such that $\beta\cat(\w+1)\prec \xi_{r+1}$. So $r<s$ but $r\ge t$. Of course $\stem(T_{\beta\cat(\w+1)})\prec A_s$; so we show that $\stem(T_{\beta\cat(\w+1)})\perp \s_k(T_\alpha)$. 

\medskip

There are two cases. First, suppose that $p=\w$, that is, $\beta\cat \w\preceq \alpha$; so $T_\alpha\succeq \Nar(T_\beta)$. At stage~$r+1$, we choose~$c$ such that $\s_c(T_\beta)$ branches off $\Nar(T_\beta)$ along $A_r$; namely, for all $\rho\in \Nar(T_\beta)$, $\rho \prec \s_c(T_\beta)\then \rho\prec A_r$. Apply this to $\rho = \s_k(T_\alpha)$: by induction, as $r\ge t$, we know that $\rho\nprec A_r$; so $\s_k(T_\alpha)\nprec \s_c(T_\beta)$. Also $\s_k(T_\alpha)\nsucceq \s_c(T_\beta)$ as the latter is not on $\Nar(T_\beta)$. Hence $\s_k(T_\alpha)\perp \s_c(T_\beta)$; and $\s_c(T_\beta)\prec \stem(T_{\beta\cat(w+1)})$. 

\smallskip

Next, suppose that $p<\w$. The only case we worry about is $p=c$; otherwise, $\s_k(T_\alpha)\succ \s_p(T_\beta)$ which is incomparable with $\s_c(T_\beta)$. So suppose that $p=c$. Then $T_\alpha\succeq T_{\beta\cat c} = \Nar(T_\beta\dbrak{c})$. Let~$w$ be the last stage at which $\beta\cat c\prec \xi_w$. Since~$\alpha$ extends $\beta\cat c$, it follows that the last stage at which~$\alpha$ is accessible is $\le w$. We conclude that $w\ge t$, and so that $\s_k(T_\alpha)\perp A_w$. Now we follow a similar argument: At stage~$r+1$ we choose~$d$ so that $\s_d(T_\beta\dbrak{c})$ branches off~$A_w$, so that for all $\rho\in \Nar(T_\beta\dbrak{c})$, if $\rho\prec \s_d(T_\beta\dbrak{c})$ then $\rho\prec A_w$. Applying this to $\rho = \s_k(T_\alpha)$, we conclude that $\s_k(T_\alpha)\nprec \s_d(T_\beta\dbrak{c}) = \stem(T_{\beta\cat(\w+1)})$. And $\s_k(T_\alpha)\in \Nar(T_\beta\dbrak{c})$, whereas $\s_d(T_\beta\dbrak{c})\notin \Nar(T_\beta\dbrak{c})$, so $\s_k(T_\alpha)\nsucceq \stem(T_{\beta\cat(\w+1)})$; overall we conclude that $\s_k(T_\alpha)\perp \stem(T_{\beta\cat(\w+1)})$ as required. 
\end{proof}

\subsubsection{The Martin-L\"of test}

Fix some $e<\w$ such that $\+P_e^A\ne\emptyset$; recall that we let $\alpha_e = \xi\rest{\seq{e,0}}$ be the shortest $e$-strategy $\prec \xi$. We show how to enumerate a higher $A$-ML test capturing~$X_e$. Note that what follows is not uniform in~$e$; the extra bit of information is the identity of~$\alpha_e$. 

\begin{lemma} \label{lemma:univ_test_help}
Let $\alpha\succeq \alpha_e$ be an $e$-strategy. Suppose that $\stem(T_\alpha)\prec A$; then $\alpha\cat(\w+1)\notin \Access$. 
\end{lemma}
\begin{proof}
By induction on the length of~$\alpha$. We know that $\alpha_e\cat (\w+1)\nprec \xi$ (\cref{lem:the_omega_plus_one_outcome}), and as the approximation $\seq{\xi_s}$ is non-decreasing, this means that $\alpha_e\cat(\w+1)\notin \Access$. 

Suppose this has been verified for all $e$-strategies $\beta\prec \alpha$; suppose that $\stem(T_\alpha)\prec A$. Suppose, for a contradiction, that $\alpha\cat(\w+1)\in \Access$; let $s$ be the stage at which $\alpha\cat(\w+1)\prec \xi_s$. By \cref{lem:the_missing_lemma}, at stage~$s$ we act for~$\alpha$: let $\beta$ be the previous $e$-strategy below~$\alpha$; at stage~$s$ we let $\beta\cat (o+1)\prec \xi_{s+1}$, while $\beta\cat o\prec \xi_s$. 
As $\stem(T_\beta)\preceq \stem(T_\alpha)\prec A$, by induction, $\beta\cat (\w+1)\notin \Access$, so $o\ne \w$, so $o<\w$. Now \cref{lemma:universal_test_main_lemma} applies to~$\beta$ and to $o+1$, which shows that $\stem(T_{\beta\cat o})\nprec A$; but $\stem(T_\alpha)\succeq \stem(T_{\beta\cat o})$, which is a contradiction.
\end{proof}

We can then deduce:

\begin{lemma} \label{lemma:final_cor_univ_test}
For any $e$-strategy $\alpha \succeq \alpha_e$ there is at most one outcome $o \leq \omega+1$ such that $\alpha\cat o\in \Access$ and $\stem(T_{\alpha\cat o})\prec A$. 
\end{lemma}

\begin{proof}
	We only need worry if $\stem(T_\alpha)\prec A$. The lemma then follows from combining \cref{lemma:univ_test_help} and \cref{lemma:universal_test_main_lemma}. 
\end{proof}

\begin{lemma} \label{lemma:univ_test_at_most_2n_strat}
For any $n$, there are at most $3^n$ many strategies $\alpha \in \Access$ of length~$n$ such that $\stem(T_\alpha) \prec A$.
\end{lemma}

\begin{proof}
Given some $\alpha$ such that $\stem(T_\alpha) \prec A$, we may have $\stem(T_{\alpha \cat \omega}) \prec A$, $\stem(T_{\alpha \cat (\omega + 1)}) \prec A$, but $\stem(T_{\alpha \cat k}) \prec A$ for at most a single~$k$. 
\end{proof}

Recall now that the pairing function we used relied on a partition of~$\w$ into blocks $[4^n,4^{n+1})$. For each such interval~$I$ we have $|I| = 3\min I$. If~$I$ is an $e$-block, and $m = \min I$, let 
\[
	C_I = \left\{ \alpha \,:\,  \alpha\succeq \alpha_e, |\alpha| = 4m, \alpha\in \Access \andd \stem(T_\alpha)\prec A \right\}.
\]
Then by \cref{lemma:univ_test_at_most_2n_strat} and \cref{lemma:final_cor_univ_test}, $|C_I|\le 3^m$. For each $\alpha\in C_I$, we have $|X_{e,\alpha}|\ge 3m$. Hence, letting 
\[
	\+U_I = \bigcup \left\{ [X_{e,\alpha}] \,:\, \alpha\in C_I  \right\},
\]
we have
\[
	\leb \left( \+U_I  \right) \le 3^m 2^{-3m} < 2^{-m}.
\]
Also, $\+U_I$ is higher $A$-c.e.\ open, uniformly in~$I$, and $X_e\in \+U_I$. Thus, to define a higher $A$-ML-test capturing~$X_e$, for each~$n$, we find an $e$-block~$I$ with $\min I\ge n$, and let $\+V_n = \+U_I$. We conclude that $X_e$ is not higher $A$-ML random. 

\subsubsection{$\omega_1^A > \wock$}

This is the last thing that remains to be proved: $\omega_1^A > \wock$. It is enough to see that $A$ hyperarithmetically computes $\xi$, because $\seq{\xi_s}$ is a collapsing approximation of~$\xi$ (or else $A$ is hyperarithmetic). Note that the computation is is not continuous, since we need to check every prefix of $A$ to know where it lies in the tree of trees and thus retrive $\xi$. So in fact, $\xi$ is higher Turing computable from $A'$ (the usual Turing jump of~$A$). Also as $\xi$ is higher left-c.e.\ and not hyperarithmetic, we have $\omega_1^\xi > \wock$ and thus $\omega_1^A > \wock$.

\section{Summary and questions}

We sum up here most of what is known about the three classes of bad oracles studied in this paper.

\renewcommand*{\arraystretch}{1.5}

\vspace{\baselineskip}

\begin{tabularx}{\textwidth}{X|X|X|X}
Bad oracles for  						&ML-random  		&$\geq_{\wock} O$ 					& $\geq_h O$						\\ \hline \hline
\multirow{2}{*}{Uniform self-PA}		&may be $^1$ /		& \multirow{2}{*}{must be $^3$}		& \multirow{2}{*}{must be $^4$} 	\\ 
										&may not be $^2$	&									& 									\\\hline
\multirow{2}{*}{Turing functionals} 	&may be $^5$/		& {may be $^7$} /					& \multirow{2}{*}{must be $^8$}		\\
										&may not be $^6$	&	may not be?						&									\\ \hline
\multirow{2}{*}{Universal ML-test} 		&\multirow{2}{*}{cannot be $^9$} 	&\multirow{2}{*}{cannot be$^{10}$} 	&may be $^{11}$ / 	\\
										& 					& 				& may not be ? \\
\end{tabularx}

\renewcommand*{\arraystretch}{1}

\vspace{\baselineskip}

\begin{enumerate}
\item By \cref{thm:bad_random_for_self_PA}. 
\item It is easy to ensure that the direct construction (\cref{thm:existence_of_bad_for_uniform_higher_self_PA}) produces a non-random sequence. 
\item By \cref{prop:no-self-PA:incomplete_case}.
\item By (4) and \cref{prop:avoiding_consistently_computable_reals}. 

\item By \cref{th:badmlrandomoracle}.
\item It is easy to ensure that the direct construction (\cref{thm:bad_oracle_for_consistent_Turing}) produces a non-random sequence. 
\item Because a bad oracle for uniform self-PA is also bad for consistent functionals (\cref{prop:avoiding_consistently_computable_reals}) and each such is $\ge_{\wock\Tur} O$ (see (3)). 
\item By \cite[Prop.2.3]{BGMContinuousHigherRandomness}. 

\item By \cref{thm:good_oracles_for_universal_tests}(c).
\item By \cref{thm:good_oracles_for_universal_tests}(a).
\item By \cref{th:no_A_unniversal_mlr_test}.
\end{enumerate}

We believe that the two folowing questions are of particular interest:

\begin{question}
Must bad oracles for Turing functionals higher compute~$O$? Are they all bad for uniform self-$\PA$?
\end{question}

\begin{question}
Must bad oracles for universal $\ML$-tests collapse $\wock$?
\end{question}

\bibliographystyle{alpha}
\bibliography{bad_oracle}

\end{document}